\documentclass[12pt,authoryear]{elsarticle}
\journal{}

\makeatletter
\def\ps@pprintTitle{%
 \let\@oddhead\@empty
 \let\@evenhead\@empty
 \def\@oddfoot{\hfill\thepage}%
 \let\@evenfoot\@oddfoot}
\makeatother

\usepackage{amsmath,amsfonts,amssymb,amsthm,graphicx}
\providecommand{\doi}[1]{\href{https://doi.org/#1}{DOI:#1}}
\usepackage{xurl} 
\renewcommand{\doi}[1]{%
 \href{https://doi.org/#1}{\nolinkurl{DOI:#1}}%
}

\usepackage{mathtools} 
\usepackage{appendix} 
\usepackage{xcolor} 
\usepackage{enumerate} 
\usepackage{dsfont} 
\usepackage{natbib} 
\usepackage{float} 
\usepackage{hyperref} 
\usepackage{subcaption} 
\usepackage{geometry} 
\theoremstyle{plain}
\newtheorem{theorem}{Theorem}
\newtheorem{proposition}{Proposition}
\newtheorem{lemma}{Lemma}

\theoremstyle{definition}

\newtheorem{remark}{Remark}

\newcommand{\N}{\mathbb{N}}
\newcommand{\R}{\mathbb{R}}
\newcommand{\PP}{\mathsf{P}}
\newcommand{\EE}{\mathsf{E}}
\newcommand{\Var}{\mathsf{Var}}

\newcommand{\rd}{\mathrm{d}}
\newcommand{\ind}{\mathds{1}}
\newcommand{\e}{\varepsilon}
\newcommand{\oo}{\mathrm{o}}
\newcommand{\OO}{\mathcal{O}}
\newcommand{\leqdef}{\vcentcolon=}
\newcommand{\reqdef}{=\vcentcolon}

\begin{document}

\begin{frontmatter}

\title{Minimax properties of gamma kernel density estimators \\ under $L^p$ loss and $\beta$-H\"older smoothness of the target}

\author[a1]{Fr\'ed\'eric Ouimet}\ead{frederic.ouimet2@uqtr.ca}

\address[a1]{3351 Bd des Forges, D\'epartement de math\'ematiques et d'informatique, Universit\'e du Qu\'ebec \`a Trois-Rivi\`eres, Trois-Rivières, QC G8Z 4M3, Camada}

\begin{abstract}
This paper considers the asymptotic behavior in $\beta$-H\"older spaces, and under $L^p$ loss, of the non-modified gamma kernel density estimator introduced by Chen ({\it Annals of the Institute of Statistical Mathematics}, 52, 471--480, 2000) for the analysis of nonnegative data, in the situation where the target may have a finite effective or true upper endpoint but the estimator itself is left untruncated and treats the support as $[0,\infty)$. The finite endpoint is an analytical device in the definition of the function class and the risk, not information supplied to the estimator. The functional classes are chosen so that the target density matches smoothly to zero at the upper endpoint, which isolates the behavior at the origin and avoids additional upper-endpoint leakage bias. The estimator attains the minimax rate for $(p,\beta)\in [1,3)\times(0,2]$ and part of $[3,4)\times(0,2]$, but fails for $p\geq 4$ or $\beta>2$, leaving a small region unresolved.
\end{abstract}


\end{frontmatter}

\section{Introduction}\label{sec:intro}

\subsection{Overview}

Let $X_1,\ldots,X_n$ be independent and identically distributed (iid) random variables taking values in $[0,\infty)$ with an unknown density $f$. Estimating $f$ nonparametrically is a classical problem, and kernel density estimators (KDEs) in the sense of \citet{Rosenblatt1956} and \citet{Parzen1962} remain a standard tool; see, e.g., \citet{Silverman1986,WandJones1995,ChaconDuong2018}. When the support of $f$ is constrained (for instance, $[0,\infty)$ or $[0,1]$), the use of symmetric kernels in the usual KDE creates a boundary bias near the endpoints, which can significantly deteriorate global $L^p$ risks. A large literature addresses this issue using reflection and pseudodata mechanisms \citep{Schuster1985,CowlingHall1996,KarunamuniAlberts2005}, transformation approaches \citep{MarronRuppert1994}, or various boundary correction techniques \citep{GasserMuller1979,Rice1984,Muller1991,Jones1993,ChengFanMarron1997,ZhangKarunamuniJones1999}.

When one wants to encode support information directly in the estimator, one may instead use \emph{asymmetric} kernels whose support reflects that information and whose shape varies with the evaluation point. Prominent contributions are due to \citet{Chen1999}, who introduced \emph{beta kernel} estimators for densities supported on $[0,1]$, and to \citet{Chen2000Gamma}, who proposed a \emph{gamma kernel} estimator for densities supported on $[0,\infty)$, together with a modified gamma version designed to improve bias behavior near the origin. In addition to density estimation, \citet{Chen2000BetaReg} developed beta-kernel smoothers for regression curves, and \citet{Chen2002LocalLinear} studied local linear regression smoothers based on asymmetric kernels, including gamma-type constructions. Among gamma-type refinements, \citet{HirukawaSakudo2015} introduced a family of generalized gamma kernels that contains Chen's modified gamma kernel as a special case and also yields Weibull and Nakagami-$m$ kernels, while \citet{IgarashiKakizawa2018} further studied generalized gamma kernel density estimators for nonnegative data, allowing positive and negative exponent parameters and developing bias-reduced versions. Other extensions and applications include the semiparametric density estimation approach of \citet{BouezmarniRombouts2009} based on copulas, applications in econometrics \citep{BouezmarniScaillet2005,FernandesGrammig2005,BouezmarniRombouts2010,HirukawaSakudo2016,SongHouZhou2019}, and many other relevant references such as \citet{FernandesMonteiro2005,FauziMaesono2020,IgarashiKakizawa2020,Some2020Bayesian,SomeEtAl2022CombinedGamma,FunkeHirukawa2025Uniform,FunkeHirukawa2025Splicing}.

Some of these techniques have notably been adapted to more complex supports such as the simplex \citep{AitchisonLauder1985,ChaconMateuFiguerasMartinFernandez2011,OuimetTolosanaDelgado2022,BertinGenestKlutchnikoffOuimet2023,GenestOuimet2025,Bouzebda2024,DaayebGenestKhardaniKlutchnikoffOuimet2025,DaayebKhardaniOuimet2025}, product spaces like the unit hypercube or positive orthant \citep{BouezmarniRombouts2009,BouezmarniRombouts2010b,FunkeKawka2015,Hirukawa2018,FunkeHirukawa2025Uniform}, half-spaces \citep{BelzileDesgagneGenestOuimet2025}, and the cone of positive definite matrices \citep{Ouimet2022a,BelzileGenestOuimetRichards2025}. Parallel to these specific developments, general frameworks for multivariate asymmetric kernels (called associated kernels) have been proposed and studied; see, e.g., \citet{KokonendjiSome2018,KokonendjiSome2021,AboubacarKokonendji2025,EsstafaKokonendjiNgo2025} and references therein.

While pointwise asymptotic expansions (bias, variance, limit distributions) for asymmetric kernel density estimators are by now well documented, their \emph{minimax} performance under integrated losses is more delicate. The difficulty stems from the fact that the smoothing induced by such kernels is spatially inhomogeneous, which complicates the control of global risk. In particular, depending on the support of the target density and on the loss under consideration, the variance or higher-order moments of the estimator may fail to be uniformly integrable. This lack of global moment control can prevent attainment of the classical minimax rate even when boundary bias is substantially reduced. This phenomenon was analyzed for beta kernel density estimators by~\citet{BertinKlutchnikoff2011}, who identified regimes of smoothness and loss for which minimax optimality holds or fails. Those conclusions were later extended to Dirichlet kernel density estimators on the simplex by \citet{BertinGenestKlutchnikoffOuimet2023}.

\subsection{Contributions}

The purpose of this paper is to establish parallel minimax and non-minimax results for the (ordinary/non-modified) gamma kernel density estimator introduced by \citet{Chen2000Gamma}. To isolate the boundary behavior at $0$ and to avoid additional tail phenomena in global $L^p$ risks, we work over H\"older-type classes of densities with a finite true or effective upper endpoint. For notational convenience this endpoint is normalized to $1$, but this normalization is only an analytical device: the estimator in \eqref{eq:estimator} does not receive the endpoint as input and continues to smooth on $[0,\infty)$. This setting is meant to describe nonnegative data for which an upper bound exists in the real world, at least effectively on the scale of measurement, but is unknown, ignored, or deliberately not built into the estimator. For such classes, with $\beta$ denoting the H\"older smoothness index and provided the H\"older radius is sufficiently large, the benchmark minimax rate under $L^p$ loss is of order $n^{-\beta/(2\beta + 1)}$; see, e.g., \citet{Tsybakov2009}. Our main results show that, with a suitable bandwidth choice, the gamma kernel density estimator achieves this rate for all $p\in[1,3)$ and $\beta\in (0,2]$, and also for a nontrivial subset of $(p,\beta)$ with $p\in [3,4)$ and $\beta\in (0,2]$; see Theorem~\ref{thm:minimax}. On the other hand, we prove two complementary non-minimaxity statements: the estimator cannot be minimax for any $p\geq 1$ when $\beta\in (2,\infty)$ (see Proposition~\ref{prop:2}), and it cannot be minimax when $p\in [4,\infty)$ even if $\beta\in (0,2]$ (see Proposition~\ref{prop:3}). These conclusions mirror those obtained for beta and Dirichlet kernel density estimators.

\subsection{Outline}

The paper is organized as follows. Section~\ref{sec:definitions} introduces the gamma kernel density estimator, the risk criteria, the definition of $\beta$-smoothness, the corresponding H\"older-type class $\Sigma(\beta,L)$ of densities relevant to our results, and other relevant notational conventions. Section~\ref{sec:practical.scope} discusses the practical scope of the compact-support formulation, including the role of the finite upper-endpoint, the practical meaning of the compatibility built into $\Sigma(\beta,L)$ (see the defining bounds in \eqref{eq:def.Sigma} below and the endpoint implication explained in Remark~\ref{rem:endpoint.matching}), and the cases where two-sided bounded-support procedures are more natural. Section~\ref{sec:main.results} states the main results, i.e., the regions of $(p,\beta)$ for which minimaxity and non-minimaxity are established. Section~\ref{sec:regularity} explains the role of the upper-endpoint compatibility condition through mirrored gamma examples. Section~\ref{sec:future} outlines several directions for future work. Proofs are gathered in Section~\ref{sec:proofs}, with proofs of some technical lemmas relegated to Appendix~\ref{app}.

\section{Definitions and notation}\label{sec:definitions}

Let $X_1,\ldots,X_n$ be a sequence of independent and identically distributed (iid) random variables with an unknown density $f$ supported on $[0,\infty)$. The goal is to estimate $f$ from the sample.

For a smoothing parameter (bandwidth) $b \in (0,\infty)$ and a point $x\in [0,\infty)$, define the gamma kernel
\begin{equation}\label{eq:kernel}
K_b(x,t) \equiv K_{x/b + 1,b}(t) = \frac{t^{x/b} e^{-t/b}}{b^{x/b + 1} \Gamma(x/b + 1)} \ind_{[0,\infty)}(t), \quad t\in [0,\infty),
\end{equation}
where $\Gamma(\cdot)$ is Euler's gamma function. The associated gamma kernel density estimator is
\begin{equation}\label{eq:estimator}
\hat{f}_{n,b}(x) = \frac{1}{n} \sum_{i=1}^n K_{x/b + 1,b}(X_i), \quad x\in [0,\infty).
\end{equation}
The family $\{\hat{f}_{n,b}: b > 0\}$ is indexed by the smoothing parameter $b = b(n)$, which is typically taken to depend on the sample size $n$.

Let $\xi_x$ be a gamma random variable with density $t\mapsto K_{x/b + 1,b}(t)$ on $[0,\infty)$. Then
\[
\EE[\hat{f}_{n,b}(x)] = \int_0^{\infty} K_{x/b + 1,b}(t) f(t) \rd t = \EE[f(\xi_x)].
\]
Moreover, for this parametrization, one has
\[
\EE(\xi_x) = x + b, \quad \Var(\xi_x) = x b + b^2.
\]
Hence, the smoothing induced by $\hat{f}_{n,b}$ is \emph{spatially inhomogeneous}: the typical fluctuation scale of $\xi_x$ around $x$ is of order $\sqrt{xb}$ when $x/b$ is large. This feature makes the choice of the functional class (and the control of global $L^p$ risks) more delicate than in compact-support settings.

Although the estimator is defined on the whole half-line, the minimax classes below consist of densities supported on $[0,1]$. This upper bound is a device for the analysis rather than an input to the estimator. In applying $\hat{f}_{n,b}$, one smooths exactly as on $[0,\infty)$; equivalently, any finite upper endpoint is treated as infinity. The estimator is not truncated or renormalized at $1$, and all kernel mass beyond the true support is kept. This is precisely why the support and differentiability requirements in the definition of $\Sigma(\beta,L)$ below matter: together, they impose the upper-endpoint compatibility discussed in Remark~\ref{rem:endpoint.matching}.

For $p\in [1,\infty)$ and a measurable function $g:[0,\infty)\to \R$, write
\[
\|g\|_p = \left(\int_0^{\infty} |g(x)|^p \rd x\right)^{1/p}.
\]
For an estimator $f_n$ of $f$, define its $L^p$ risk at $f$ by
\[
R_n(f_n,f) = \left\{\EE\big(\|f_n - f\|_p^p\big)\right\}^{1/p},
\]
whenever the expectation exists. For a class of densities $\mathcal{F}$, define the maximal risk
\[
R_n(f_n,\mathcal{F}) = \sup_{f\in \mathcal{F}} R_n(f_n,f),
\]
and the minimax risk
\[
r_n(\mathcal{F}) = \inf_{f_n} R_n(f_n,\mathcal{F}),
\]
where the infimum is over all estimators based on $(X_1,\ldots,X_n)$.

For $\beta\in (0,\infty)$, let
\[
m = \sup\{\ell\in \N_0 : \ell < \beta\},
\]
and define the $\beta$-H\"older class $\Sigma(\beta,L)$ as the set of all densities $f$ supported on $[0,1]$ that are $m$-times differentiable on $(0,\infty)$ (so a jump discontinuity is allowed at $0$ but not at $1$) and such that
\begin{equation}\label{eq:def.Sigma}
\max_{0 \leq k \leq m} \sup_{u\in (0,\infty)}|f^{(k)}(u)| \leq L \quad \text{and} \quad
\sup_{\substack{u,v\in (0,\infty) \\ u\neq v}} \frac{|f^{(m)}(u) - f^{(m)}(v)|}{|u-v|^{\beta-m}} \leq L,
\end{equation}
where $f^{(k)}$ denotes the $k$th derivative of $f$ (with $f^{(0)} = f$).

\begin{remark}
Here, the definition of $\Sigma(\beta,L)$ is stated for densities supported on $[0,1]$ for simplicity. Indeed, a new space, say $\smash{\widetilde{\Sigma}(\beta,L)}$, could be defined analogously for densities supported on $[0,C]$ for $C\in (0,\infty)$. All results in the paper would still be valid after the usual rescaling: if $\tilde{f}$ is supported on $[0,C]$, then $f(u)=C\tilde{f}(Cu)$ is supported on $[0,1]$, and conversely $\tilde{f}(x)=C^{-1}f(x/C)$, with the H\"older constant adjusted by a finite factor depending only on $C$ and $\beta$. This normalization is only a theoretical convention for the class and the risk; the estimator itself remains the untruncated gamma kernel estimator on the positive half-line and does not require the value of $C$.
\end{remark}

The main question is whether the family $\{\hat{f}_{n,b}: b > 0\}$ can achieve the minimax rate over $\Sigma(\beta,L)$ under $L^p$ loss, for an appropriate choice of $b = b_n$.

Throughout the paper, expectation is taken with respect to the joint law of the mutually independent copies $X_1,\ldots,X_n$ of $X$. The notation $u = \OO(v)$ means that $\limsup |u / v| < B < \infty$ as $n \to \infty$ or $b \to 0$, depending on the context. The positive constant $B$ may depend on the risk exponent $p$, the smoothness parameter $\beta$, the Lipschitz constant $L$, and the target density $f$, but not on any other variable unless explicitly written as a subscript. Similarly, throughout the proofs, $c,C\in (0,\infty)$ denote generic positive constants whose values may change from expression to expression and which may depend on $p$, $\beta$, $L$, and $f$, but not on $n$ or $b$. If both $u = \OO(v)$ and $v = \OO(u)$ hold, then one writes $u \asymp v$. Similarly, the notation $u = \oo(v)$ means that $\lim |u / v| = 0$ as $n\to \infty$ or $b\to 0$. Subscripts indicate which parameters the convergence rate can depend on. If $f_n$ is any estimator of $f$, then $\EE[f_n]$ is a shorthand for the map $x\mapsto \EE[f_n(x)]$. The gamma distribution always has the shape/scale parametrization.

\section{Practical scope of the compact-support formulation}\label{sec:practical.scope}

\subsection{What is encoded in the estimator}

The compact-support formulation should be understood as an analytical framework for the risk, not as additional support information supplied to the estimator. The endpoint $1$ is a normalization used to define the class $\Sigma(\beta,L)$ and to carry out the global $L^p$ analysis. The estimator studied throughout the paper is the ordinary untruncated and unrenormalized gamma kernel estimator in \eqref{eq:estimator}, defined on the whole half-line $[0,\infty)$. Consequently, even when the target density has a finite true or effective upper endpoint, the smoothing rule itself treats that endpoint as infinity. This is precisely the situation considered here: the finite endpoint is part of the data-generating mechanism and is used for the risk analysis, but it is not built into the estimator.

This distinction is important in applications. Many nonnegative quantities are bounded in the real world by physical, administrative, technological, biological, or measurement constraints, even though the value of the bound may be unavailable, uncertain, unstable, or irrelevant to the way the estimator is implemented. For example, human body heights are naturally modeled as nonnegative measurements, and although an absolute upper bound exists for biological and physical reasons, that bound is not usually specified or encoded when applying a half-line smoothing method. In such situations one may use a method designed for nonnegative data without specifying a right endpoint. The present results describe when this practice is harmless at the first-order minimax-rate level. They do not assert that gamma kernels are generally preferable for bounded-support problems. Rather, they identify conditions under which the ordinary gamma estimator retains the classical compact-support minimax rate although it ignores the finite upper endpoint.

\subsection{Endpoint compatibility}

The key compatibility requirement comes from the definition of $\Sigma(\beta,L)$: densities are supported on $[0,1]$, while the required derivatives are defined and bounded on $(0,\infty)$. In particular, the target density, and the relevant derivatives when $\beta>1$, must match the zero extension to the right of the finite endpoint (here $1$). Since the gamma kernels in \eqref{eq:estimator} have support on $[0,\infty)$, they can place mass to the right of the finite endpoint. If the density remains positive near the endpoint, or if it has a sharp cutoff there, this mismatch can create an additional leakage contribution to the global risk. The matching condition rules out this behavior by forcing the target density to approach the upper endpoint smoothly enough. Its technical role is detailed in Remark~\ref{rem:endpoint.matching}.

The formal restrictions on the loss exponent $p$ and the smoothness parameter $\beta$ are part of the minimax statement and are therefore left to the main results below in Section~\ref{sec:main.results}. Their practical implication is that the same estimator can behave differently depending on how strongly the loss weights local deviations and how much smoothness one wants to exploit. The remarks in the main results section separate this stochastic issue from the endpoint-compatibility issue.

\subsection{When other bounded-support procedures are preferable}

If the exact two-sided support is truly known, for example, $[0,1]$, and the goal is to encode both endpoints in the estimator, then beta kernels or other compact-support procedures are generally more natural because their support matches the support of the target density. Likewise, if the density is expected to have a sharp upper cutoff or to remain positive at the right endpoint, the ordinary gamma estimator should not be viewed as the default choice. In that case, estimators that use the upper endpoint, such as beta kernels, boundary-corrected compact-support kernels, truncated gamma kernels, or renormalized gamma kernels, are more directly adapted to the problem. Accordingly, the results below should be interpreted as rate-optimality and non-rate-optimality statements for the ordinary half-line gamma estimator when the right endpoint exists but is unknown, ignored, or not built into the smoothing rule, not as a general recommendation to use gamma kernels whenever the support is bounded.

\section{Main results}\label{sec:main.results}

For every fixed $\beta\in (0,\infty)$ and $p\in [1,\infty)$, provided $L$ is sufficiently large depending on $\beta$, the minimax rate over $\Sigma(\beta,L)$ under $L^p$ loss is
\begin{equation}\label{eq:minimax.rate}
r_n\{\Sigma(\beta,L)\} \asymp n^{-\beta/(2\beta + 1)};
\end{equation}
see, e.g., \citet[Theorem~5.1]{IbragimovHasminskii1981} and \citet[Theorem~2.8]{Tsybakov2009}.

The following theorem states a minimax property of the gamma kernel density estimator \eqref{eq:estimator} when the bandwidth is tuned according to the smoothness parameter $\beta$.

\begin{theorem}\label{thm:minimax}
Let $L > 0$ be sufficiently large, depending on $\beta$. Define
\[
\mathcal{S} = \left\{(p,\beta)\in [3,4)\times(0,2] : \frac{p-3}{p-2} < \beta \leq 2\right\}.
\]
Assume that $(p,\beta)\in [1,3)\times(0,2]$ or that $(p,\beta)\in \mathcal{S}$. Let $b_n = c \, n^{-2/(2\beta + 1)}$ for all $n\in \N$ and some constant $c\in (0,\infty)$. Then
\[
\limsup_{n\to \infty} \frac{R_n\{\hat{f}_{n,b_n}, \Sigma(\beta,L)\}}{r_n\{\Sigma(\beta,L)\}} < \infty,
\]
i.e., the sequence $\{\hat{f}_{n,b_n}: n\in \N\}$ achieves the minimax rate over $\Sigma(\beta,L)$ under $L^p$ loss.
\end{theorem}

\begin{remark}\label{rem:endpoint.matching}
The somewhat restrictive definition of $\Sigma(\beta,L)$ (smoothness on $(0,\infty)$ even though $\mathrm{supp}(f)\subseteq[0,1]$) implicitly enforces a compatibility condition at the upper endpoint $x=1$: since $f(x)=0$ for $x>1$ and the derivatives are assumed to exist and be bounded on $(0,\infty)$, one necessarily has $f(1)=0$ and, when $\beta>1$, also $f'(1)=0$, etc. This condition should be viewed as the compatibility needed to study a half-line gamma smoother on a compact risk class, not as support information available to the smoother. It prevents an additional ``endpoint leakage'' bias coming from the fact that the gamma kernel has support on $[0,\infty)$ and is not truncated at~$1$.

To see what would happen without this restriction, imagine working instead with the more standard H\"older-type class, say $\Sigma^\circ(\beta,L)$, of densities supported on $[0,1]$ that are $\beta$-H\"older on $(0,1)$ but with \emph{no} matching condition at $x=1$ (so that, for instance, $f(1^-)$ may be strictly positive). In Step~4 of the proof of Theorem~\ref{thm:minimax}, H\"older regularity and, when $\beta > 1$, a first-order Taylor expansion yield the bound $|\EE[f(\xi_x)] - f(x)| = \OO(b^{\beta/2})$ for $\beta\in (0,2]$ under the upper-endpoint compatibility condition. If $f$ is allowed to have a nonzero left limit at $1$, then for $x<1$ close to $1$, we must also account for the event $\{\xi_x>1\}$, on which $f(\xi_x)=0$ while $f(x)$ can be of order one. A simple decomposition is
\[
|\EE[f(\xi_x)] - f(x)| \leq \EE\left(|f(\xi_x) - f(x)| \ind_{\{\xi_x\leq 1\}}\right) + f(x) \, \PP(\xi_x > 1).
\]
For $x < 1$, the first term is $\OO(b^{(\beta \wedge 1)/2})$, whereas the second term can be of constant order near $1$. When $\beta > 1$, the sharper bound in Step~4 uses cancellation of the linear Taylor term in the full expectation, which is lost in the absolute-value bound above. For $x>1$, one has $f(x)=0$, and the first term itself can be of constant order. Thus, when $x$ lies within the kernel's typical fluctuation scale of the endpoint, i.e. $|1-x| = \OO(b^{1/2})$, the pointwise bias can be $\asymp 1$ on an interval of length $\asymp b^{1/2}$ around $1$ (on both sides of~$1$ when we integrate over $[0,\infty)$). This yields an additional integrated bias contribution of order
\[
\left(\int_{1 - c \, b^{1/2}}^{1 + c \, b^{1/2}} |\EE[\hat f_{n,b}(x)] - f(x)|^p \, \rd x\right)^{1/p} \asymp b^{1/(2p)}.
\]
For $\beta\in (0,2]$, the endpoint contribution $b^{1/(2p)}$ is negligible relative to $b^{\beta/2}$ when $\beta < 1/p$, has the same order when $\beta = 1/p$, and dominates it when $\beta > 1/p$. These comparisons concern the bias only. Conclusions about the risk additionally require control of the stochastic term at the chosen bandwidth, including the higher-moment contribution near the origin. Thus, this bias comparison alone does not establish an optimal bandwidth or risk rate over $\Sigma^\circ(\beta,L)$.
\end{remark}

The following results identify regimes of the loss exponent $p$ and smoothness parameter $\beta$ for which the gamma kernel density estimator $\hat{f}_{n,b}$ defined in~\eqref{eq:estimator} cannot achieve the minimax rate over the $\beta$-H\"older class $\Sigma(\beta,L)$, irrespective of the choice of bandwidth sequence $(b_n)_{n\in \N}$.

\begin{proposition}\label{prop:2}
Let $p\in [1,\infty)$ and $\beta\in (2,\infty)$ be given. There exists $L > 1$ such that for every bandwidth sequence $(b_n)_{n\in \N}\subseteq (0,1)$, the family $\{\hat{f}_{n,b_n}:n\in \N\}$ satisfies
\[
\liminf_{n\to \infty}\frac{R_n\{\hat{f}_{n,b_n},\Sigma(\beta,L)\}}{r_n\{\Sigma(\beta,L)\}} = + \infty.
\]
\end{proposition}

\begin{proposition}\label{prop:3}
Let $p\in [4,\infty)$ and $\beta\in (0,2]$ be given. There exists $L > 1$ such that for every bandwidth sequence $(b_n)_{n\in \N}\subseteq (0,1)$, the family $\{\hat{f}_{n,b_n}:n\in \N\}$ satisfies
\[
\liminf_{n\to \infty} \frac{R_n\{\hat{f}_{n,b_n},\Sigma(\beta,L)\}}{r_n\{\Sigma(\beta,L)\}} = + \infty.
\]
\end{proposition}

\begin{remark}
The practical message of Theorem~\ref{thm:minimax} is conditional and should be separated from the case of a fully known compact support. The theorem does not say that the gamma kernel estimator is the preferred procedure whenever the support is bounded. If the support is truly known to be $[0,1]$ and this information should be encoded in the estimator, beta kernels are often the natural asymmetric-kernel choice because they match both endpoints. The value of the present result is instead to describe the ordinary gamma estimator on $[0,\infty)$ for nonnegative data whose true or effective support is finite, but whose upper endpoint is unknown, ignored, or treated as infinity by the smoothing procedure. In this sense, the finite endpoint is a tool for the risk analysis rather than an input to the estimator. For losses with $p<3$, the minimax condition covers the full range $\beta\in(0,2]$, and for $p\in[3,4)$ it covers the smoother part of the range, namely $\beta>(p-3)/(p-2)$. Thus the restriction is mild for the most common integrated losses, but it becomes more stringent as one gives more weight to large local deviations. The theorem says that, when the density tapers smoothly enough at the unseen upper bound, ignoring that bound does not change the first-order minimax rate in those regions. When the density remains positive, has a sharp cutoff, or has a nonzero slope at the upper endpoint, the endpoint-leakage term mentioned in the previous remark can dominate the usual smoothing bias; in that case, beta kernels, truncated or renormalized gamma kernels, or other compact-support procedures are preferable if the bounded support is known. Likewise, when $\beta > 2$ and the goal is to exploit the faster corresponding minimax rate, Proposition~\ref{prop:2} suggests that the ordinary gamma kernel should be replaced by a method designed to exploit higher smoothness, such as higher-order bias-corrected asymmetric kernel estimators on the semi-infinite interval \citep[e.g.,][]{IgarashiKakizawa2020}, bias-reduced generalized gamma kernels \citep[e.g.,][]{IgarashiKakizawa2018}, or local polynomial density estimators with automatic boundary adaptation \citep[e.g.,][]{CattaneoJanssonMa2020}, rather than used only through the slower guarantee associated with a smaller smoothness index.
\end{remark}

\begin{remark}
Theorem~\ref{thm:minimax} and Propositions~\ref{prop:2} and~\ref{prop:3} leave open the region
\[
\left\{(p,\beta): 3 < p < 4,\ 0 < \beta \leq \frac{p-3}{p-2}\right\}.
\]
This is a genuine limitation of the present analysis rather than a statement that the gamma kernel density estimator is, or is not, minimax in that region. In the proof of Theorem~\ref{thm:minimax}, the condition $\beta>(p-3)/(p-2)$ is used to dominate the first Rosenthal term in the integrated stochastic bound near the origin. By contrast, the available non-minimax lower bounds cover different regimes: Proposition~\ref{prop:2} treats all $p\in[1,\infty)$ when $\beta>2$, while Proposition~\ref{prop:3} treats $p\geq 4$ when $\beta\in(0,2]$. Settling the remaining region would require either sharper integrated moment bounds for the gamma kernel estimator or a new lower-bound argument tailored to this intermediate range.
\end{remark}

\section{Illustration of the upper-endpoint smoothness matching condition}\label{sec:regularity}

The purpose of this section is illustrative rather than prescriptive. We do not propose the family in \eqref{eq:mirrored-gamma} below as a general modeling class for bounded data, nor do we suggest that gamma kernels are preferable to beta kernels when the two endpoints are known. Instead, we use it to make the upper-endpoint smoothness matching condition in the definition of $\Sigma(\beta,L)$ visible on an explicit family. The endpoint $x=1$ should again be read as a theoretical device that makes a true or effective upper bound visible, while the estimator itself treats the right endpoint as infinity. Many bounded-support families, including beta-type densities with sufficiently high powers at the endpoints, could be checked in a similar way. We focus on mirrored gamma densities because the paper concerns gamma-kernel smoothing, because the behavior near the upper endpoint is governed by a single familiar shape parameter, and because the example directly displays the compatibility constraint created by applying a kernel supported on $[0,\infty)$ to a target supported on $[0,1]$.

The critical technical condition under which Theorem~\ref{thm:minimax} is established is the assumption that the underlying univariate density $f$ belongs to the $\beta$-H\"older space $\Sigma(\beta,L)$ for appropriate choices of the smoothness parameter $\beta\in(0,2]$ and the Lipschitz constant $L\in(0,\infty)$. In view of the definition of $\Sigma(\beta,L)$, this entails not only compact support on $[0,1]$, but also a smooth matching at the upper endpoint $x=1$, since $f(x)=0$ for $x>1$ and the derivatives are required to exist and be bounded on $(0,\infty)$. To provide a concrete illustration of these constraints, we consider target densities obtained by mirroring a gamma density around $1$ and truncating to~$[0,1]$.

Let $g=g_{\alpha,\theta}$ be a gamma density with shape parameter $\alpha\in(0,\infty)$ and scale parameter $\theta\in(0,\infty)$, given by
\begin{equation}\label{eq:gamma-density}
g_{\alpha,\theta}(s) = \frac{s^{\alpha - 1} e^{-s/\theta}}{\theta^{\alpha} \Gamma(\alpha)} \ind_{(0,\infty)}(s).
\end{equation}
Define the normalizing constant
\[
c_{\alpha,\theta} \leqdef \left(\int_0^1 g_{\alpha,\theta}(u) \, \rd u\right)^{-1}\in (1,\infty),
\]
and the associated \emph{mirrored gamma density truncated to $[0,1]$} by
\begin{equation}\label{eq:mirrored-gamma}
f_{\alpha,\theta}(x) \leqdef c_{\alpha,\theta} \, g_{\alpha,\theta}(1-x)\ind_{[0,1]}(x), \qquad x\in [0,\infty);
\end{equation}
see Figure~\ref{fig:mirrortruncgamma}.

\begin{figure}[!ht]
\centering
\begin{subfigure}[t]{0.49\textwidth}
\centering
\includegraphics[width=\textwidth]{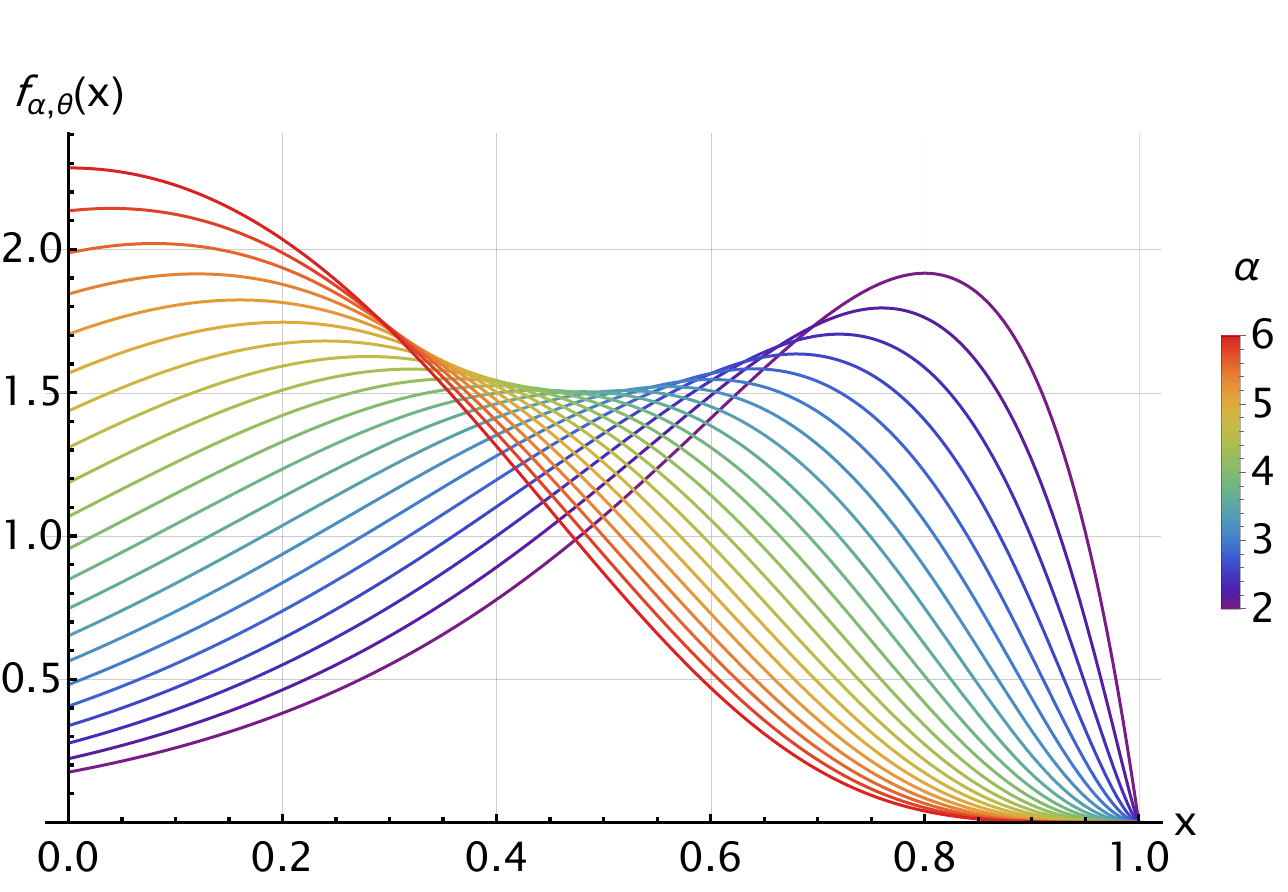}
\end{subfigure}
\hfill
\begin{subfigure}[t]{0.49\textwidth}
\centering
\includegraphics[width=\textwidth]{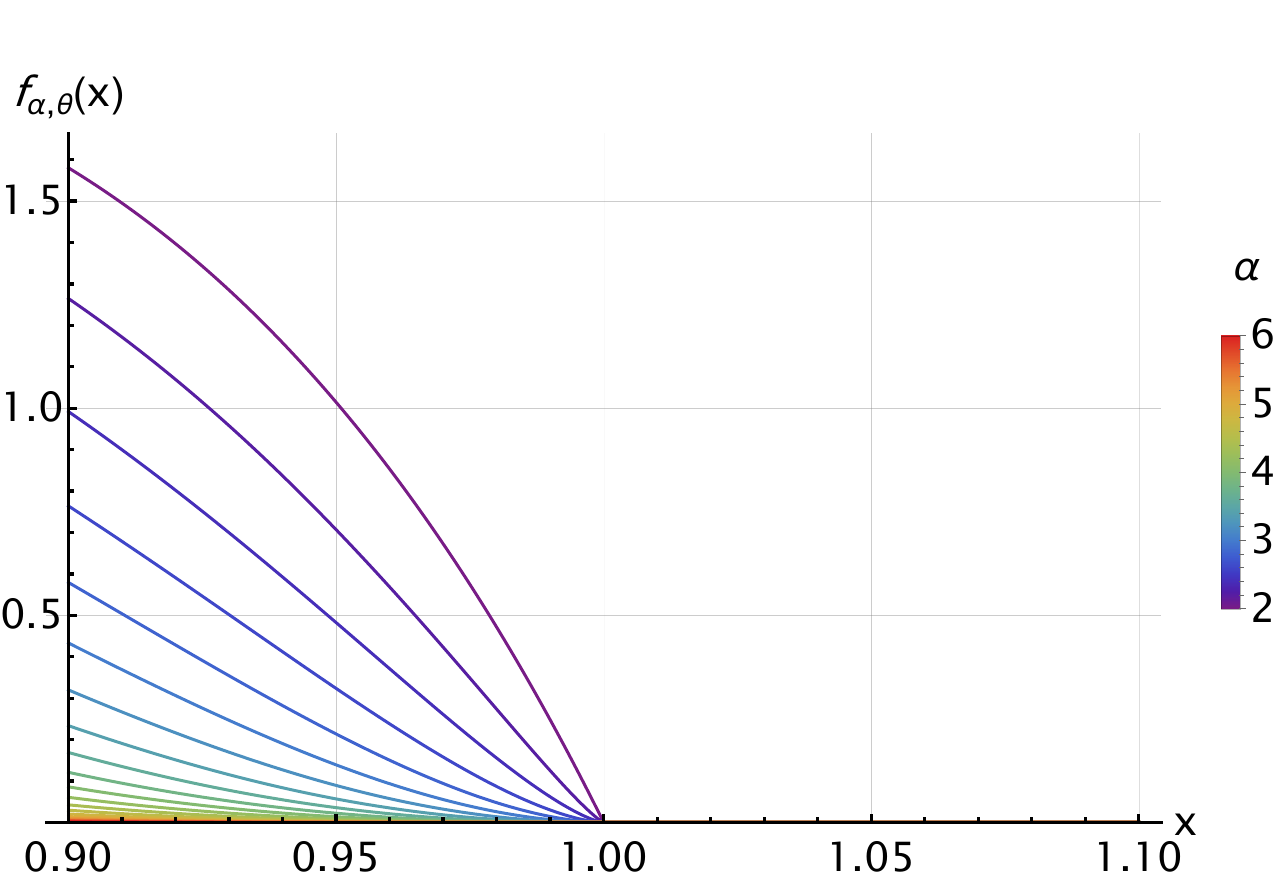}
\end{subfigure}
\caption{Visualization of the mirrored gamma densities truncated to $[0,1]$ defined in \eqref{eq:mirrored-gamma} for $\theta=0.2$ and various values of the shape parameter $\alpha$ (from $2.0$ to $6.0$ with $0.2$ increments). The left panel shows the full range $x\in[0,1]$, while the right panel zooms in on $x\in[0.9,1.1]$.}
\label{fig:mirrortruncgamma}
\end{figure}

By construction, $f_{\alpha,\theta}$ is supported on $[0,1]$ and integrates to one. Moreover, since $g_{\alpha,\theta}$ is $C^{\infty}$ on $(0,\infty)$, it follows that $f_{\alpha,\theta}$ is $C^{\infty}$ on $(0,1)$. For every $k\in\N_0$ and every $x\in (0,1)$,
\begin{equation}\label{eq:mirrored-derivatives}
\smash{f_{\alpha,\theta}^{(k)}}(x) = c_{\alpha,\theta}(-1)^k g_{\alpha,\theta}^{(k)}(1-x).
\end{equation}
On the other hand, $f_{\alpha,\theta}(x)=0$ for $x>1$, and hence $\smash{f_{\alpha,\theta}^{(k)}}(x)=0$ for all $k\in\N_0$ and all $x>1$. Therefore, the only potential obstruction to $f_{\alpha,\theta}\in\Sigma(\beta,L)$ arises at the endpoint $x=1$, and it is governed by the behavior of $\smash{g_{\alpha,\theta}^{(k)}}(s)$ as $s\downarrow 0$.

To describe this behavior, note that as $s\downarrow 0$,
\[
g_{\alpha,\theta}(s) = \frac{s^{\alpha-1}}{\theta^{\alpha}\Gamma(\alpha)}\left\{1 + \OO(s)\right\}.
\]
Differentiating $s^{\alpha-1}e^{-s/\theta}$ repeatedly shows that for each fixed $k\in\N_0$, there exists a function $h_{\alpha,\theta,k}$ that is bounded on $(0,1]$ (and depends only on $\alpha,\theta,k$) such that
\begin{equation}\label{eq:gamma-derivative-factor}
g_{\alpha,\theta}^{(k)}(s) = s^{\alpha-1-k} \, h_{\alpha,\theta,k}(s), \qquad s\in (0,1].
\end{equation}
Combining \eqref{eq:mirrored-derivatives} and \eqref{eq:gamma-derivative-factor} yields that, for every $k\in\N_0$,
\begin{equation}\label{eq:mirrored-endpoint-rate}
\smash{f_{\alpha,\theta}^{(k)}}(x) = \OO\bigl((1-x)^{\alpha-1-k}\bigr), \qquad x\uparrow 1.
\end{equation}
In particular, if $\alpha>k+1$, then $\smash{f_{\alpha,\theta}^{(k)}}(x)\to 0$ as $x\uparrow 1$, which ensures the upper-endpoint smoothness matching condition $\smash{f_{\alpha,\theta}^{(k)}}(1)=0$ required for $k$th differentiability at $x=1$ (given that $\smash{f_{\alpha,\theta}^{(k)}}(x)\equiv 0$ for $x>1$). Conversely, if $\alpha\leq k+1$, then $\smash{f_{\alpha,\theta}^{(k)}}$ fails to vanish at $x=1$ (and may even diverge), preventing membership in $\Sigma(\beta,L)$ once derivatives up to order $k$ are required.

These observations lead to a simple characterization of the largest H\"older smoothness index that can be accommodated by the family $\{f_{\alpha,\theta}:\alpha,\theta>0\}$.

\begin{proposition}\label{prop:regularity}
Let $g_{\alpha,\theta}$ be the gamma density \eqref{eq:gamma-density} and let $f_{\alpha,\theta}$ be its mirrored-and-truncated version \eqref{eq:mirrored-gamma}.
\begin{itemize}\setlength\itemsep{0em}
\item[{\rm (i)}]
If $\alpha>1$, then for every $\beta\in (0,\alpha-1]$ there exists $L\in(0,\infty)$ (depending on $\alpha,\theta,\beta$) such that $f_{\alpha,\theta}\in \Sigma(\beta,L)$.
\item[{\rm (ii)}]
If $\beta>\alpha-1$, then for every $L\in(0,\infty)$ one has $f_{\alpha,\theta}\notin \Sigma(\beta,L)$.
\end{itemize}
\end{proposition}

The proof is straightforward from \eqref{eq:mirrored-endpoint-rate}. For~(i), let $\beta\in(0,\alpha-1]$ and let $m=\sup\{\ell\in\N_0:\ell<\beta\}$. Then $\beta-m\in(0,1]$ and $\alpha-1-m\geq \beta-m>0$, so \eqref{eq:mirrored-endpoint-rate} implies that $\smash{f_{\alpha,\theta}^{(k)}}(x)\to 0$ as $x\uparrow 1$ for all $0\leq k\leq m$, and that $\smash{f_{\alpha,\theta}^{(m)}}(x)=\OO((1-x)^{\alpha-1-m})$ near~$1$. Since the map $t\mapsto t^{\alpha-1-m}$ is $(\beta-m)$-H\"older on $[0,1]$ whenever $\alpha-1-m\geq \beta-m$, it follows that $\smash{f_{\alpha,\theta}^{(m)}}$ is $(\beta-m)$-H\"older in a neighborhood of~$1$; away from~$1$, the function is smooth and hence (locally) Lipschitz, which is stronger than $(\beta-m)$-H\"older because $\beta-m\leq 1$. This yields $f_{\alpha,\theta}\in\Sigma(\beta,L)$ for some finite~$L$ after taking suprema over the compact support. For~(ii), if $\alpha\leq m+1$ (with $m$ associated to $\beta$), then $f_{\alpha,\theta}$ fails to be $m$-times differentiable at $x=1$ by \eqref{eq:mirrored-endpoint-rate}. If instead $\alpha>m+1$, then $\smash{f_{\alpha,\theta}^{(m)}}(1)=0$ but, writing $x=1-h$ with $h\downarrow 0$, \eqref{eq:mirrored-endpoint-rate} gives
\[
|f_{\alpha,\theta}^{(m)}(1-h) - f_{\alpha,\theta}^{(m)}(1)| \asymp h^{\alpha - 1 - m}
\quad\Longrightarrow\quad
\frac{|f_{\alpha,\theta}^{(m)}(1-h) - f_{\alpha,\theta}^{(m)}(1)|}{h^{\beta - m}} \asymp h^{\alpha - 1 - \beta},
\]
so the H\"older quotient with exponent $\beta-m$ diverges whenever $\beta>\alpha-1$.

Finally, note that the inclusion property
\begin{equation}\label{eq:inclusion}
f \in \Sigma(\beta,L) \quad \Longrightarrow \quad \forall_{\beta^{\star}\in(0,\beta)} \exists_{L^{\star}\in(0,\infty)} \quad f \in \Sigma(\beta^{\star},L^{\star})
\end{equation}
holds in full generality. In particular, if one believes that the target density has smoothness $\beta>2$, then Theorem~\ref{thm:minimax} can still be invoked by working with a lower smoothness index $\beta^{\star}\in(0,2]$ for which $(p,\beta^{\star})$ belongs to the minimaxity region of Theorem~\ref{thm:minimax}, and for which the target also belongs to $\Sigma(\beta^{\star},L^{\star})$. This comes at the expense of using the slower minimax rate $n^{-\beta^{\star}/(2\beta^{\star}+1)}$ associated with the larger class $\Sigma(\beta^{\star},L^{\star})$. Proposition~\ref{prop:2} shows that the ordinary gamma kernel density estimator cannot exploit smoothness above order $2$ in a minimax way over $\Sigma(\beta,L)$. Thus, when smoothness $\beta>2$ is credible and the faster rate $n^{-\beta/(2\beta+1)}$ is the statistical target, the ordinary gamma estimator should be viewed, in the admissible lower-smoothness regimes just described, only as providing a robust lower-smoothness guarantee. To pursue the faster rate while still treating the right endpoint as unknown or not built into the smoothing rule, one should instead consider higher-order bias-corrected asymmetric kernel estimators on the semi-infinite interval \citep[e.g.,][]{IgarashiKakizawa2020}, bias-reduced generalized gamma kernels \citep[e.g.,][]{IgarashiKakizawa2018}, or local polynomial density estimators with automatic boundary adaptation \citep[e.g.,][]{CattaneoJanssonMa2020}. If reliable information about a finite upper endpoint is available and is meant to be encoded in the estimator, compact-support or truncated and renormalized variants may also be considered, but that is a different modeling choice from the ordinary half-line gamma estimator studied here.

\section{Discussion and future research directions}\label{sec:future}

\subsection{Adaptive bandwidth selection, endpoint leakage, and the remaining minimax gap}

The present work treats the iid setting and assumes that the smoothness index $\beta$ (hence the bandwidth order) is known. A natural continuation is to devise \emph{fully data-driven} choices of $b$ that remain rate-optimal for the global $L^p$ risk, and to clarify how the spatially inhomogeneous smoothing of gamma kernels impacts oracle-type inequalities and adaptation. Related Goldenshluger--Lepski-type selection ideas and oracle inequalities have been developed for local polynomial density estimation on complicated supports, including domains with local pinches; see, e.g., \citet{Bertin2025}. Closely related is the problem of moving beyond the analytical compact-support/matching assumption at the upper endpoint. If the upper endpoint is known and should be built into the estimator, beta kernels or truncated and renormalized gamma kernels are the relevant competitors. If the endpoint is unknown but the data are effectively bounded, it would be useful to develop risk bounds that separate upper-endpoint leakage from the lower-boundary behavior isolated in the present paper.

An important question is the remaining strip $\left\{(p,\beta): 3 < p < 4,\ 0 < \beta \leq (p-3)/(p-2)\right\}$, where the present paper gives neither minimaxity nor non-minimaxity. Resolving it would require a sharper analysis of the stochastic term near the origin or a new lower-bound construction. This is important because the gap is not caused by the bias calculation at the upper endpoint, but by the integrated moments of the gamma kernel near the lower endpoint.

\subsection{Modified and generalized gamma kernels}

It would also be natural to extend the present minimax analysis to Chen's modified gamma kernel \citep{Chen2000Gamma} and to generalized gamma kernels \citep{HirukawaSakudo2015,IgarashiKakizawa2018}. For the uncorrected variants, the existing pointwise and MISE calculations suggest that the same minimax and non-minimax conclusions should be robust, up to changes in constants and under uniform versions of the required moment and tail bounds. Indeed, Chen's modified gamma kernel changes the local parametrization near the origin and removes the first-derivative term from the leading interior bias, but its leading bias is still of order $b$, its interior variance has the same order $n^{-1}b^{-1/2}x^{-1/2}f(x)$, and its boundary variance has the same order $n^{-1}b^{-1}$ as for the ordinary gamma kernel. Thus the proof of Theorem~\ref{thm:minimax} should carry over after verifying the corresponding kernel power bounds, and the non-minimaxity mechanisms in Propositions~\ref{prop:2} and~\ref{prop:3} should persist because they are driven by order-$b$ bias, by the usual interior stochastic fluctuation, and, for large $p$, by the near-origin singularity of the integrated stochastic term, not by the exact constants in the kernel.

The same reasoning applies to the uncorrected generalized gamma kernels of \citet{HirukawaSakudo2015} and to the uncorrected generalized gamma estimator of \citet{IgarashiKakizawa2018}. These kernels were designed to have the same first-order bias and variance structure as the modified gamma kernel, and their kernel powers have the same interior and boundary orders needed in the Rosenthal-type bounds used in the proofs. A complete proof would nevertheless require checking these estimates uniformly in the evaluation point and for the relevant $L^p$ powers, together with the tail bounds needed when the target is supported on $[0,1]$ but the estimator is still supported on $[0,\infty)$. This is especially important for the generalized gamma kernels with negative exponents in \citet{IgarashiKakizawa2018}, where finite-moment restrictions depend on the order of the moment being used. The upper-endpoint smoothness matching condition would remain necessary for all these untruncated half-line estimators, since none of them encodes the finite endpoint unless it is explicitly truncated or renormalized.

The bias-reduced estimators of \citet{IgarashiKakizawa2018} are different. Since their bias is of order $b^2$ rather than $b$, they should not be expected to satisfy the same non-minimaxity statement for $\beta>2$. They are designed precisely to exploit higher smoothness, and their global minimax analysis would require a separate treatment, with different smoothness ranges and possibly different restrictions on $p$. In particular, one would need to control the signed or multiplicatively corrected kernels, the higher-order bias terms, the boundary stochastic moments, and the same upper-endpoint leakage effect considered in the present paper.

\subsection{Other extensions}

Another interesting direction is to extend the minimax analysis to other asymmetric-kernel smoothing problems, such as nonparametric regression or conditional density estimation with nonnegative responses, building on the gamma/beta kernel regression and local-linear constructions of \citet{Chen2000BetaReg,Chen2002LocalLinear}. It would also be valuable to understand to what extent the regions of $(p,\beta)$ identified here persist when the sample is \emph{dependent} (e.g., stationary strongly mixing sequences), where the stochastic term must reflect long-run variance contributions and the analysis requires dependence-adapted moment inequalities.

Finally, extending global $L^p$ minimax results to higher-dimensional associated kernels \citep[e.g.,][]{KokonendjiSome2021} and to matrix supports remains largely open. In particular, obtaining analogues of Theorem~\ref{thm:minimax} for the Wishart kernel density estimator on the cone of symmetric positive definite matrices \citep{BelzileGenestOuimetRichards2025} appears promising, but the underlying geometry and boundary structure make the control of bias and integrated moments substantially more challenging.

\section{Proofs}\label{sec:proofs}

\subsection{Proof of Theorem~\ref{thm:minimax}}\label{sec:proof.thm.1}

\textbf{Step 1: Risk decomposition.}
Fix $n\in \N$, $b\in (0,1]$ and $f\in \Sigma(\beta,L)$. Decompose
\[
\hat{f}_{n,b} - f = (\hat{f}_{n,b} - \EE[\hat{f}_{n,b}]) + (\EE[\hat{f}_{n,b}] - f).
\]
By applying the triangle inequality for the $L^p([0,\infty))$ and $L^p(\Omega)$ norms, respectively, we have
\begin{equation}\label{eq:decomp}
\begin{aligned}
\big\{\EE(\|\hat{f}_{n,b} - f\|_p^p)\big\}^{1/p}
&\leq \big\{\EE\big[\big(\|\hat{f}_{n,b} - \EE[\hat{f}_{n,b}]\|_p + \|\EE[\hat{f}_{n,b}] - f\|_p\big)^p\big]\big\}^{1/p} \\
&\leq \big\{\EE(\|\hat{f}_{n,b} - \EE[\hat{f}_{n,b}]\|_p^p)\big\}^{1/p} + \|\EE[\hat{f}_{n,b}] - f\|_p \\
&\equiv A_n(b,f) + B_n(b,f).
\end{aligned}
\end{equation}
Taking suprema over $f\in \Sigma(\beta,L)$ yields
\[
R_n\{\hat{f}_{n,b},\Sigma(\beta,L)\} \leq \sup_{f\in \Sigma(\beta,L)} A_n(b,f) + \sup_{f\in \Sigma(\beta,L)} B_n(b,f).
\]

\bigskip
\textbf{Step 2: Kernel bounds.}
For the next steps, we register bounds below on the kernel: (i) a uniform bound; (ii) an $L^2([0,\infty))$ bound; and (iii) an exponential tail bound (valid for $x\geq 3$).
\begin{enumerate}[(a)]
\item For $x > 0$, the gamma density $t\mapsto K_b(x,t)$ is unimodal with mode at $t = x$, so $\|K_b(x,\cdot)\|_{\infty} = K_b(x,x)$. By Stirling's lower bound
\begin{equation}\label{eq:Stirling.bound}
\Gamma(u + 1) \geq \sqrt{2\pi} \, u^{u + 1/2} e^{-u}, \quad u > 0,
\end{equation}
\citep[see, e.g.,][]{Batir2008} with $u = x/b$, we have
\[
K_b(x,x) = \frac{(x/b)^{x/b} e^{-x/b}}{b \, \Gamma(x/b + 1)} \leq \frac{1}{\sqrt{2\pi}} \, b^{-1/2} x^{-1/2}, \quad x > 0.
\]
Fix $x\in (0,1]$ and $b\in (0,1]$. If $x\geq b$, then $x+b\leq 2x$ and thus $x^{-1/2}\leq \sqrt{2} \, (x+b)^{-1/2}$, which implies
\[
K_b(x,x) \leq \frac{1}{\sqrt{\pi}} \, b^{-1/2}(x+b)^{-1/2}.
\]
If instead $0<x<b$, write $u=x/b\in (0,1)$. Using the integral representation $\Gamma(u+1)=\int_0^{\infty} t^{u}e^{-t}\rd t$ and restricting to $[u,u+1]$, we get
\[
\Gamma(u+1)\geq \int_u^{u+1} t^{u}e^{-t}\rd t \geq \int_u^{u+1} u^{u}e^{-(u+1)}\rd t
= u^{u}e^{-(u+1)},
\]
hence
\[
K_b(x,x)=\frac{1}{b} \, \frac{e^{-u}u^{u}}{\Gamma(u+1)}\leq \frac{e}{b}.
\]
Since $x<b$ implies $x+b\leq 2b$, we have $b^{-1}\leq \sqrt{2} \, b^{-1/2}(x+b)^{-1/2}$, and therefore
\[
K_b(x,x)\leq e\sqrt{2} \, b^{-1/2}(x+b)^{-1/2}.
\]
Combining the previous cases yields that, for all $x\in (0,1]$ and $b\in (0,1]$,
\begin{equation}\label{eq:sup-local}
\|K_b(x,\cdot)\|_{\infty} \leq C \, b^{-1/2}(x + b)^{-1/2}.
\end{equation}
In particular, $\|K_b(x,\cdot)\|_{\infty} \leq C \, b^{-1}$ for all $x\in (0,1]$.

\item A direct calculation yields
\begin{equation}\label{eq:Bb.def}
\int_0^{\infty} K_b(x,t)^2 \rd t = \frac{b^{-1} \, \Gamma(2x/b + 1)}{2^{2x/b + 1} \Gamma(x/b + 1)^2} \reqdef B_b(x).
\end{equation}
By Stirling's formula, it follows that
\begin{equation}\label{eq:Bb}
B_b(x) \leq C \, b^{-1/2}(x + b)^{-1/2}, \quad x\in (0,1], ~b\in (0,1].
\end{equation}

\item Fix $b\in (0,1]$ and $x \geq 3$. Since the gamma density $t\mapsto K_b(x,t)$ is unimodal and the mode is at $t = x \geq 1$, it is increasing on $[0,1]$. Hence
\[
\sup_{t\in [0,1]} K_b(x,t) = K_b(x,1) = \frac{e^{-1/b}}{b^{x/b + 1} \Gamma(x/b + 1)}.
\]
Applying Stirling's lower bound $\Gamma(u+1)\geq \sqrt{2\pi} \, u^{u+1/2}e^{-u}$ with $u=x/b$ gives
\[
K_b(x,1)
\leq \frac{C}{\sqrt{x b}}\exp\left(-\frac{x\ln x - x + 1}{b}\right).
\]
Moreover, the map $x\mapsto (x\ln x - x + 1)/x = \ln x - 1 + 1/x$ is increasing for $x\geq 1$, hence for all $x\geq 3$,
\[
x\ln x - x + 1 \geq c^{\star} x,
\qquad c^{\star} \leqdef \ln 3 - 1 + \tfrac{1}{3} > 0.
\]
Therefore, there exist constants $c^{\star},C > 0$ such that
\begin{equation}\label{eq:exp-tail}
\sup_{t\in [0,1]} K_b(x,t) \leq \frac{C}{\sqrt{x b}} \exp(-c^{\star} x/b), \quad x \geq 3, ~b\in (0,1].
\end{equation}
\end{enumerate}

\bigskip
\textbf{Step 3: Control of the stochastic term $A_n(b,f)$.}
Write
\[
Z_b(x) \leqdef \hat{f}_{n,b}(x) - \EE[\hat{f}_{n,b}(x)] = \frac{1}{n} \sum_{i=1}^n \{K_b(x,X_i) - \EE[K_b(x,X_i)]\}.
\]
Since $\mathrm{supp}(f)\subseteq [0,1]$, we have $X_i\in [0,1]$ almost surely; hence
\begin{equation}\label{eq:Mb}
|K_b(x,X_i) - \EE[K_b(x,X_i)]| \leq 2 M_b(x), \quad M_b(x) \leqdef \sup_{t\in [0,1]} K_b(x,t).
\end{equation}
Moreover,
\[
\Var\bigl(K_b(x,X_1)\bigr) \leq \EE[K_b(x,X_1)^2] \leq \|f\|_{\infty} \int_0^1 K_b(x,t)^2 \rd t \leq L B_b(x),
\]
where we used $\|f\|_{\infty} \leq L$ and \eqref{eq:Bb.def}.

For $p\in [2,\infty)$, we use Rosenthal's inequality \citep[see, e.g.,][Theorem~3]{Rosenthal1970}: if $Y_1,\ldots,Y_n$ are iid and centered, $|Y_1| \leq M$ almost surely and $\Var(Y_1) = \sigma^2$, then
\[
\EE\left(\left|\frac{1}{n} \sum_{i=1}^n Y_i\right|^p\right) \leq C_p\left\{\left(\frac{M}{n}\right)^{p-2}\frac{\sigma^2}{n} + \left(\frac{\sigma^2}{n}\right)^{p/2}\right\}.
\]
Applying this with $Y_i = K_b(x,X_i) - \EE[K_b(x,X_i)]$, $M = 2M_b(x)$ and $\sigma^2 = \Var(K_b(x,X_1))$ yields, for $p \geq 2$,
\begin{equation}\label{eq:pointwise-moment}
\EE(|Z_b(x)|^p) \leq C_p\left\{\left(\frac{M_b(x)}{n}\right)^{p-2}\frac{L B_b(x)}{n} + \left(\frac{L B_b(x)}{n}\right)^{p/2}\right\}.
\end{equation}
For $p\in [1,2)$, we simply use Lyapunov's inequality
\begin{equation}\label{eq:Lyapunov}
\EE(|Z_b(x)|^p) \leq (\EE[Z_b(x)^2])^{p/2} = (\Var(\hat{f}_{n,b}(x)))^{p/2} \leq \left(\frac{L B_b(x)}{n}\right)^{p/2},
\end{equation}
so \eqref{eq:pointwise-moment} remains valid with the first term dropped inside the braces.

Now, we want to integrate $\EE(|Z_b(x)|^p)$ over $x\geq 0$. We split the integral as follows:
\[
\int_0^{\infty} \EE(|Z_b(x)|^p) \rd x
= \int_0^1 \EE(|Z_b(x)|^p) \rd x
+ \int_1^3 \EE(|Z_b(x)|^p) \rd x
+ \int_3^{\infty} \EE(|Z_b(x)|^p) \rd x.
\]

\medskip
\emph{Integral over $[0,1]$.}
For $x\in (0,1]$, we have $M_b(x) \leq \|K_b(x,\cdot)\|_{\infty}$ and thus \eqref{eq:sup-local} gives $M_b(x) \leq C \, b^{-1/2}(x + b)^{-1/2}$; moreover $B_b(x) \leq C \, b^{-1/2}(x + b)^{-1/2}$ by \eqref{eq:Bb}. For $p \geq 2$ and any $q\in [0,1]$, interpolate the bounds $M_b(x) \leq C \, b^{-1}$ and $M_b(x) \leq C \, b^{-1/2}(x + b)^{-1/2}$ to obtain
\begin{equation}\label{eq:interp}
M_b(x) \leq C \, b^{-1 + q/2}(x + b)^{-q/2}, \quad x\in (0,1].
\end{equation}
Combining \eqref{eq:pointwise-moment}, \eqref{eq:Bb} and \eqref{eq:interp} gives, for $p \geq 2$,
\begin{equation}\label{eq:An-integral-01}
\begin{aligned}
&\int_0^1 \EE(|Z_b(x)|^p) \rd x \\
&\quad\leq C \left[n^{-(p-1)} b^{-\frac{(p-2)(2-q) + 1}{2}} \int_0^1 (x + b)^{-\frac{(p-2)q + 1}{2}} \rd x + n^{-p/2} b^{-p/4} \int_0^1 (x + b)^{-p/4} \rd x\right].
\end{aligned}
\end{equation}
The second integral on the right-hand side is finite uniformly in $b\in (0,1]$ provided $p < 4$. The first integral on the right-hand side is finite provided
\begin{equation}\label{eq:q-integrability}
\frac{(p-2)q + 1}{2} < 1 \quad \Longleftrightarrow \quad (p-2) q < 1.
\end{equation}
If $p\in [1,2)$, then \eqref{eq:Lyapunov} and \eqref{eq:Bb} give directly
\[
\int_0^1 \EE(|Z_b(x)|^p) \rd x \leq C n^{-p/2} b^{-p/4} \int_0^1 (x+b)^{-p/4} \rd x \leq C n^{-p/2}b^{-p/4},
\]
since $p/4<1$.

\medskip
\emph{Integral over $[1,3]$.}
For $x\in [1,3]$, the global supremum satisfies $M_b(x)\leq \|K_b(x,\cdot)\|_\infty = K_b(x,x)\leq C b^{-1/2}$ (since $x^{-1/2}\leq 1$ on $[1,3]$), and similarly $B_b(x)\leq C b^{-1/2}$ (e.g., from \eqref{eq:Bb.def} and Stirling's formula). Plugging these bounds in \eqref{eq:pointwise-moment} yields, for $p\geq 2$,
\[
\EE(|Z_b(x)|^p) \leq C_p\Big\{n^{-(p-1)} b^{-(p-1)/2} + n^{-p/2} b^{-p/4}\Big\},
\quad x\in[1,3],
\]
and for $p\in[1,2)$, \eqref{eq:Lyapunov} gives $\EE(|Z_b(x)|^p)\leq C n^{-p/2}b^{-p/4}$. Consequently, for all $p\in[1,\infty)$,
\begin{equation}\label{eq:integral.1.3}
\int_1^3 \EE(|Z_b(x)|^p) \, \rd x
\leq C\Big\{n^{-(p-1)} b^{-(p-1)/2} \ind_{[2,\infty)}(p) + n^{-p/2} b^{-p/4}\Big\}.
\end{equation}

\medskip
\emph{Tail integral over $[3,\infty)$.}
For the tail part $x \geq 3$, we use \eqref{eq:exp-tail} and \eqref{eq:Mb}. Since $0\leq K_b(x,X_1)\leq M_b(x)$ almost surely, we have
\[
\Var\bigl(K_b(x,X_1)\bigr) \leq M_b(x)^2.
\]
If $p\geq 2$, then Rosenthal's inequality gives
\[
\EE(|Z_b(x)|^p) \leq C_p\Big\{n^{-(p-1)}M_b(x)^p + n^{-p/2}M_b(x)^p\Big\} \leq C_p n^{-p/2}M_b(x)^p.
\]
If $p\in [1,2)$, then Lyapunov's inequality gives
\[
\EE(|Z_b(x)|^p) \leq \left(\frac{M_b(x)^2}{n}\right)^{p/2} = n^{-p/2}M_b(x)^p.
\]
Consequently, for all $p\in [1,4)$,
\begin{equation}\label{eq:integral.3.inf}
\begin{aligned}
\int_3^{\infty} \EE(|Z_b(x)|^p) \rd x
&\leq C n^{-p/2} \int_3^{\infty} M_b(x)^p \rd x \\
&\leq C n^{-p/2} b^{-p/2} \int_3^{\infty} x^{-p/2} \exp(-c^{\star} p x/b) \, \rd x \\
&\leq C n^{-p/2} b^{1-p/2}\exp(-3c^{\star} p/b) \\
&\leq C n^{-p/2} b^{-p/4},
\end{aligned}
\end{equation}
where the last inequality uses $p<4$.

\medskip
\emph{Conclusion for $A_n(b,f)$.}
If $p\in [1,2)$, then only the variance-type term contributes and we obtain
\[
\int_0^{\infty} \EE(|Z_b(x)|^p) \rd x \leq C \, n^{-p/2}b^{-p/4}.
\]
It follows that, for $p\in [1,2)$,
\begin{equation}\label{eq:A.n.1}
A_n(b,f) \leq C \, n^{-1/2}b^{-1/4}, \quad \text{for all } f\in \Sigma(\beta,L).
\end{equation}

If $p\in [2,3)$, take $q = 1$ in \eqref{eq:An-integral-01}; then \eqref{eq:q-integrability} holds and both integrals are finite on the right-hand side of \eqref{eq:An-integral-01}. Combining the bounds on $[0,1]$, $[1,3]$ and $[3,\infty)$ yields
\[
\int_0^{\infty} \EE(|Z_b(x)|^p) \rd x
\leq C\Big\{n^{-(p-1)} b^{-(p-1)/2} + n^{-p/2}b^{-p/4}\Big\}.
\]
For the bandwidth choice $b = b_n = c \, n^{-2/(2\beta + 1)}$ in Step~5 below, one has $n b_n^{1/2}\to\infty$, and therefore
\[
\frac{n^{-(p-1)} b_n^{-(p-1)/2}}{n^{-p/2} b_n^{-p/4}}
= (n b_n^{1/2})^{-(p-2)/2} \longrightarrow 0
\quad \text{if }p>2,
\]
(and equals $1$ if $p=2$). Hence, for $p\in [2,3)$,
\begin{equation}\label{eq:A.n.2}
A_n(b_n,f) \leq C \, n^{-1/2}b_n^{-1/4}, \quad \text{for all } f\in \Sigma(\beta,L).
\end{equation}

Now let $p\in [3,4)$ and assume $(p,\beta)\in \mathcal{S}$, i.e., $\beta > (p-3)/(p-2)$. Choose $q$ such that
\[
1 - \beta < q < \frac{1}{p-2},
\]
which is possible precisely because $\beta > (p-3)/(p-2)$. With this choice, the condition in \eqref{eq:q-integrability} holds. Moreover, for $b = b_n = c \, n^{-2/(2\beta + 1)}$, we have
\[
n^{-(p-1)} b_n^{-\frac{(p-2)(2-q) + 1}{2}}
\asymp \bigl(n^{-p/2}b_n^{-p/4}\bigr) \, n^{1 - p/2} b_n^{\frac{(p-2)(2q-3)}{4}}
\asymp \bigl(n^{-p/2}b_n^{-p/4}\bigr) \, n^{1 - p/2 + \frac{(p-2)(3-2q)}{2(2\beta + 1)}}.
\]
The exponent of $n$ in the last display is
\[
1 - \frac{p}{2} + \frac{(p-2)(3-2q)}{2(2\beta + 1)} \leq 1 - \frac{p}{2} + \frac{(p-2)(1 + 2\beta)}{2(2\beta + 1)} = 0,
\]
where we used $q > 1 - \beta$, i.e., $3-2q < 1 + 2\beta$. Hence the first term on the right-hand side of \eqref{eq:An-integral-01} is dominated by the second, uniformly in $n$. Together with the bounds \eqref{eq:integral.1.3} and \eqref{eq:integral.3.inf}, we get, for $p\in [3,4)$,
\begin{equation}\label{eq:A.n.3}
A_n(b_n,f) \leq C \, n^{-1/2}b_n^{-1/4}, \quad \text{for all } f\in \Sigma(\beta,L).
\end{equation}
Combining \eqref{eq:A.n.1}, \eqref{eq:A.n.2}, and \eqref{eq:A.n.3}, and taking the supremum over $f\in \Sigma(\beta,L)$ yields, for all $(p,\beta)$ satisfying the assumptions of Theorem~\ref{thm:minimax},
\begin{equation}\label{eq:An - final}
\sup_{f\in \Sigma(\beta,L)} A_n(b_n,f) \leq C \, n^{-1/2}b_n^{-1/4}.
\end{equation}

\bigskip
\textbf{Step 4: Control of the bias term $B_n(b,f)$.}
Recall that, for each $x > 0$,
\[
\EE[\hat{f}_{n,b}(x)] = \int_0^1 K_b(x,t) f(t) \rd t = \EE[f(\xi_x)],
\]
where $\xi_x\sim \mathrm{Gamma}(x/b + 1,b)$. Hence $B_n(b,f) = \| \EE f(\xi_\cdot) - f(\cdot)\|_p$.

\medskip
\emph{Bias on $(0,3]$.}
If $\beta\in (0,1]$, then $|f(t) - f(x)| \leq L|t-x|^{\beta}$ and therefore
\[
|\EE[f(\xi_x)] - f(x)| \leq L \, \EE(|\xi_x - x|^{\beta}).
\]
If $\beta\in (1,2]$, then by Taylor's theorem with remainder and the definition of $\Sigma(\beta,L)$,
\[
f(\xi_x) - f(x) = f'(x)(\xi_x - x) + R_x, \quad |R_x| \leq L |\xi_x - x|^{\beta},
\]
so, using $|f'(x)| \leq L$, we obtain
\[
|\EE[f(\xi_x)] - f(x)| \leq L |\EE(\xi_x - x)| + L \, \EE(|\xi_x - x|^{\beta}) = L \, b + L \, \EE(|\xi_x - x|^{\beta}).
\]
In both cases, it remains to bound $\EE(|\xi_x - x|^{\beta})$. Since $\EE(\xi_x) = x + b$ and
\[
\Var(\xi_x) = x b + b^2 \leq 4 b, \quad x\in (0,3], ~b\in (0,1],
\]
Lyapunov's inequality gives $\EE(|\xi_x - \EE\xi_x|^{\beta}) \leq C \, b^{\beta/2}$. Hence, by the triangle inequality,
\[
\EE(|\xi_x - x|^{\beta}) \leq C \, \left(\EE(|\xi_x - \EE(\xi_x)|^{\beta}) + b^{\beta}\right) \leq C \, b^{\beta/2}.
\]
Therefore, for all $\beta\in (0,2]$ and all $x\in (0,3]$,
\[
|\EE[f(\xi_x)] - f(x)| \leq C \, b^{\beta/2},
\]
and consequently
\begin{equation}\label{eq:bias-01}
\int_0^3 |\EE[f(\xi_x)] - f(x)|^p \rd x \leq C \, b^{p\beta/2}.
\end{equation}

\medskip
\emph{Bias on $[3,\infty)$.}
For $x\geq 3$, $f(x)=0$ and
\[
0 \leq \EE[\hat{f}_{n,b}(x)] = \int_0^1 K_b(x,t)f(t) \rd t \leq \|f\|_{\infty} \int_0^1 K_b(x,t) \rd t \leq L M_b(x).
\]
Using \eqref{eq:exp-tail} and integrating yields
\[
\int_3^{\infty} |\EE[\hat{f}_{n,b}(x)]|^p \rd x
\leq C \, b^{-p/2}\int_3^{\infty} x^{-p/2} e^{-c^{\star} p x/b} \, \rd x
\leq C \, b^{1-p/2}e^{-3c^{\star} p/b},
\]
which is negligible compared to any power of $b$.

Combining the two regions $[0,3]$ and $[3,\infty)$ with \eqref{eq:bias-01} yields
\begin{equation}\label{eq:Bn - final}
\sup_{f\in \Sigma(\beta,L)} B_n(b,f) \leq C \, b^{\beta/2}.
\end{equation}

\bigskip
\textbf{Step 5: Choice of $b_n$ and conclusion.}
From \eqref{eq:decomp}, \eqref{eq:An - final} and \eqref{eq:Bn - final}, for the choice $b = b_n = c \, n^{-2/(2\beta + 1)}$, we obtain
\[
R_n\{\hat{f}_{n,b_n},\Sigma(\beta,L)\} \leq C (n^{-1/2}b_n^{-1/4} + b_n^{\beta/2}) \asymp n^{-\beta/(2\beta + 1)},
\]
which is the minimax rate stated in \eqref{eq:minimax.rate}. Therefore,
\[
\limsup_{n\to \infty} \frac{R_n\{\hat{f}_{n,b_n}, \Sigma(\beta,L)\}}{r_n\{\Sigma(\beta,L)\}} < \infty.
\]
This completes the proof of Theorem~\ref{thm:minimax}.

\subsection{Proof of Proposition~\ref{prop:2}}\label{sec:proof.prop.2}

For a density $f$ supported on $[0,1]$, write $\PP_f$, $\EE_f$ and $\Var_f$ for probability, expectation and variance computed under the joint law of $(X_1,\ldots,X_n)$ with iid marginals of density $f$.

Let $L > 1$. We begin with a simple device to lower bound the $L^p([0,\infty))$ norm by an $L^1([0,1])$ norm. For any measurable function $g$,
\begin{equation}\label{eq:Lp.L1}
\|g\|_p \geq \left(\int_0^1 |g(x)|^p \rd x\right)^{1/p} \geq \int_0^1 |g(x)| \rd x \equiv \|g\|_{L^1([0,1])},
\end{equation}
so by Jensen's inequality,
\begin{equation}\label{eq:Lp-to-L1}
R_n(\hat{f}_{n,b},f) = \left\{\EE_f(\|\hat{f}_{n,b} - f\|_p^p)\right\}^{1/p} \geq \EE_f(\|\hat{f}_{n,b} - f\|_{L^1([0,1])}).
\end{equation}

We first introduce the following simple (non-smooth) test functions:
\begin{equation}\label{eq:test-densities}
f_0(x) = \ind_{[0,1]}(x), \qquad f_3(x) = \{1 + \e \cdot(2x-1)\} \ind_{[0,1]}(x),
\end{equation}
where
\[
\e = \e(L) \leqdef \min\left\{\frac{1}{2},\frac{L-1}{2}\right\}\in (0,1/2];
\]
see Figure~\ref{fig:test-densities}. Then $f_0$ and $f_3$ are nonnegative and integrate to one on $[0,1]$. However, due to the hard cutoff at $x=1$, they do not satisfy the smoothness requirements imposed by the definition of $\Sigma(\beta,L)$.

\begin{figure}[!ht]
\centering
\begin{subfigure}[t]{0.49\textwidth}
\centering
\includegraphics[width=\textwidth]{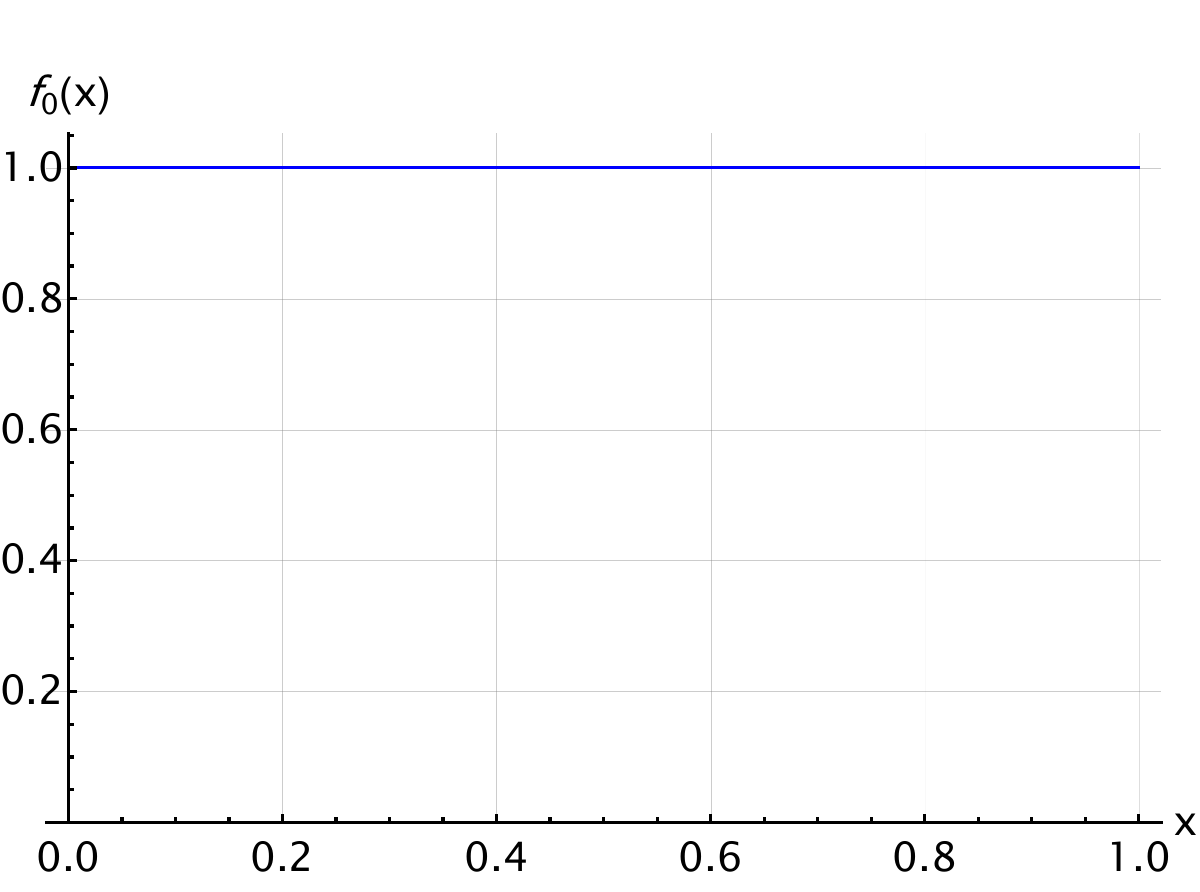}
\end{subfigure}
\hfill
\begin{subfigure}[t]{0.49\textwidth}
\centering
\includegraphics[width=\textwidth]{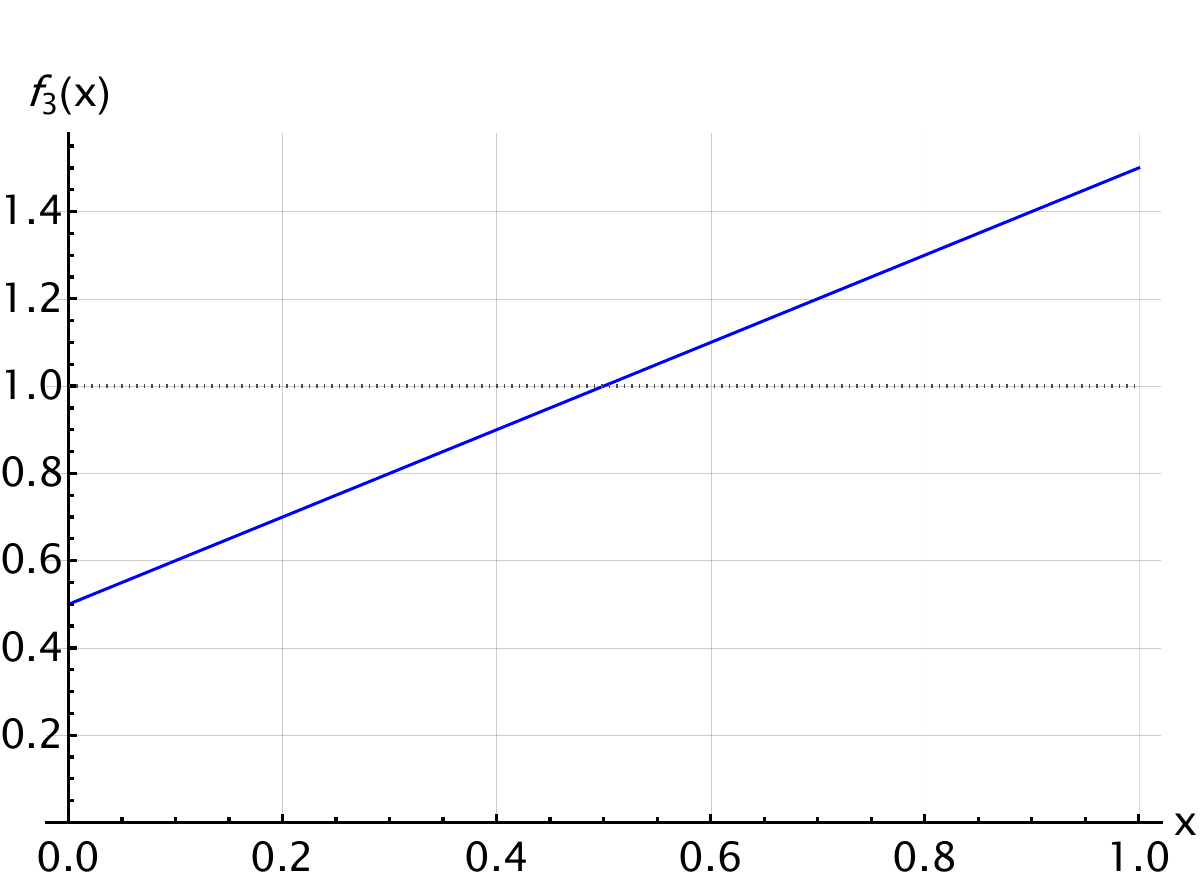}
\end{subfigure}
\caption{Visualization of the non-smooth test densities $f_0$ and $f_3$ defined in \eqref{eq:test-densities} with $L=2$ (hence $\e=1/2$).}
\label{fig:test-densities}
\end{figure}

Let $\widetilde{f}_0$ and $\widetilde{f}_3$ denote $C^{\infty}$ modifications of $f_0$ and $f_3$ on $[7/8,1]$ such that:
\begin{itemize}\setlength\itemsep{0em}
\item $\widetilde{f}_j = f_j$ on $[0,7/8]$ for $j\in\{0,3\}$;
\item $\widetilde{f}_j$ is supported on $[0,1]$ and $\int_0^1 \widetilde{f}_j(x) \, \rd x=1$;
\item $\widetilde{f}_j$ is $C^{\infty}$ on $(0,\infty)$ with $\widetilde{f}_j^{(k)}(1)=0$ for all $k\in\N_0$;
\item $\widetilde{f}_0,\widetilde{f}_3\in \Sigma(\beta,\widetilde{L})$ for an appropriate $\widetilde{L} > 1$.
\end{itemize}
This is a standard mollification construction and we omit the details, since all subsequent lower bounds are obtained by integrating over $x$ in subsets of $[0,1/2]$ and by using exponential tail bounds for the gamma kernel; the smoothing on $[7/8,1]$ does not affect the conclusions, up to exponentially small error terms that are absorbed in the constants.

The next lemma provides the two lower bounds (bias and stochastic fluctuation) needed to prove Proposition~\ref{prop:2}.

\begin{lemma}\label{lem:prop2-key}
Let $\widetilde{f}_0$ and $\widetilde{f}_3$ be as above. There exist constants $b_0\in (0,1)$ and $c_0,c_1,c_2\in (0,\infty)$ (depending at most on $\widetilde{L}$) such that the following statements hold for all $b\in (0,b_0]$.
\begin{enumerate}[(a)]
\item \textbf{Bias lower bound at $\widetilde{f}_3$.} One has
\[
\|\EE_{\widetilde{f}_3}[\hat{f}_{n,b}] - \widetilde{f}_3\|_{L^1([0,1])} \geq c_1 b.
\]
\item \textbf{Fluctuation lower bound at $\widetilde{f}_0$.} If in addition $n b^{1/2} \geq c_0$, then
\[
\EE_{\widetilde{f}_0}\left[\int_{1/4}^{1/2}\big|\hat{f}_{n,b}(x) - \EE_{\widetilde{f}_0}[\hat{f}_{n,b}(x)]\big| \rd x\right] \geq c_2 n^{-1/2}b^{-1/4}.
\]
Moreover,
\[
\sup_{x\in [1/4,1/2]}\big|\EE_{\widetilde{f}_0}[\hat{f}_{n,b}(x)] - \widetilde{f}_0(x)\big| \leq \exp(-c_0/b).
\]
\end{enumerate}
\end{lemma}

\begin{proof}[Proof of Lemma~\ref{lem:prop2-key}]
\textbf{Proof of $(a)$.}
Fix $x\in [0,1/2]$ and let $\xi_x\sim\mathrm{Gamma}(x/b + 1,b)$, so that
\[
\EE_{\widetilde{f}_3}[\hat{f}_{n,b}(x)] = \int_0^1 K_b(x,t)\widetilde{f}_3(t) \, \rd t = \EE[\widetilde{f}_3(\xi_x)].
\]
Since $\widetilde{f}_3=f_3$ on $[0,7/8]$ and $x\in[0,1/2]\subseteq[0,7/8]$, we have $\widetilde{f}_3(x)=f_3(x)$. Moreover,
\begin{equation}\label{eq:tilde-f3-error}
\big|\EE[\widetilde{f}_3(\xi_x)] - \EE[f_3(\xi_x)]\big|
= \big|\EE[(\widetilde{f}_3-f_3)(\xi_x)\ind_{\{\xi_x>7/8\}}]\big|
\leq \|\widetilde{f}_3-f_3\|_{\infty} \, \PP(\xi_x>7/8).
\end{equation}
A Chernoff bound (as in \eqref{eq:tail.bound.xi.x}, with threshold $7/8$ instead of $1$) yields
\[
\PP(\xi_x>7/8)\leq C\exp(-c/b)
\]
uniformly over $x\in[0,1/2]$ for some constants $c,C>0$. Hence, the right-hand side of \eqref{eq:tilde-f3-error} is $\OO(\exp(-c/b))$ uniformly in $x\in[0,1/2]$.
Therefore it suffices to derive a lower bound for $\EE[f_3(\xi_x)]-f_3(x)$, since the above error term is negligible as $b\to 0$.

Since $f_3(t) = \{1 + \e \cdot (2t-1)\} \ind_{[0,1]}(t)$, we have
\[
\EE[f_3(\xi_x)] = \PP(\xi_x \leq 1) + \e\left\{2 \EE(\xi_x \ind_{\{\xi_x \leq 1\}}) - \PP(\xi_x \leq 1)\right\}.
\]
Hence, using $f_3(x) = 1 + \e \cdot (2x-1)$ and the identity $\EE(\xi_x) = x + b$, we obtain
\begin{align*}
\EE[f_3(\xi_x)] - f_3(x)
&= (1 - \e) \, \{\PP(\xi_x \leq 1)-1\} + 2 \e\{\EE(\xi_x \ind_{\{\xi_x \leq 1\}})-x\} \\
&= -(1 - \e) \, \PP(\xi_x > 1) + 2 \e\left\{b - \EE(\xi_x \ind_{\{\xi_x > 1\}})\right\} \\
&\geq 2 \e b -2 \e \, \EE(\xi_x \ind_{\{\xi_x > 1\}}) - \PP(\xi_x > 1).
\end{align*}
We now bound the tail probability $\PP(\xi_x > 1)$ uniformly over $x\in [0,1/2]$. For $\lambda = 1/(2 b)$ (so that $\lambda < 1/b$, which implies the finiteness of the moment generating function), Chernoff's inequality yields
\begin{equation}\label{eq:tail.bound.xi.x}
\PP(\xi_x > 1)
\leq e^{-\lambda}\EE(e^{\lambda\xi_x}) = e^{-1/(2 b)}(1 - b\lambda)^{-(x/b + 1)}
= e^{-1/(2 b)} 2^{x/b + 1} \leq 2 \exp\left(-\frac{1 - \ln 2}{2 b}\right).
\end{equation}
Next, by Cauchy-Schwarz,
\[
\EE(\xi_x \ind_{\{\xi_x > 1\}}) \leq \{\EE(\xi_x^2)\}^{1/2} \PP(\xi_x > 1)^{1/2}.
\]
Using $\EE(\xi_x^2) = \Var(\xi_x) + \{\EE(\xi_x)\}^2 = (xb + b^2) + (x + b)^2 \leq C$ for $x\in [0,1/2]$ and $b \leq 1$, we get $\EE(\xi_x \ind_{\{\xi_x > 1\}}) \leq C \, \exp(-c/b)$ for some $c > 0$. Therefore, for $b$ small enough,
\[
\EE[f_3(\xi_x)] - f_3(x) \geq \e b, \quad x\in [0,1/2].
\]
Combining this with \eqref{eq:tilde-f3-error} yields, for $b$ small enough,
\[
\EE_{\widetilde{f}_3}[\hat{f}_{n,b}(x)] - \widetilde{f}_3(x)
= \EE[\widetilde{f}_3(\xi_x)] - f_3(x)
\geq \e b - C\exp(-c/b) \geq \frac{\e}{2}b,
\quad x\in[0,1/2].
\]
Integrating over $x\in [0,1/2]$ yields
\[
\|\EE_{\widetilde{f}_3}[\hat{f}_{n,b}] - \widetilde{f}_3\|_{L^1([0,1])}
\geq \int_0^{1/2} \big|\EE_{\widetilde{f}_3}[\hat{f}_{n,b}(x)] - \widetilde{f}_3(x)\big| \rd x
\geq \frac{\e}{4} b,
\]
which is $(a)$ with $c_1 = \e/4$.

\bigskip
\textbf{Proof of $(b)$.}
Fix $x\in [1/4,1/2]$. Under $\widetilde{f}_0$, the variable $\zeta_1(x) \leqdef K_b(x,X_1)$ satisfies
\[
\EE_{\widetilde{f}_0}[\zeta_1(x)] = \int_0^1 K_b(x,t)\widetilde{f}_0(t) \, \rd t, \quad
\EE_{\widetilde{f}_0}[\zeta_1(x)^2] = \int_0^1 K_b(x,t)^2\widetilde{f}_0(t) \, \rd t.
\]
Let $Y_1(x) = \zeta_1(x) - \EE_{\widetilde{f}_0}[\zeta_1(x)]$ and set $\sigma_x^2 \leqdef \Var_{\widetilde{f}_0}\{Y_1(x)\} = \Var_{\widetilde{f}_0}\{\zeta_1(x)\}$. Also let
\[
Z_n(x) \leqdef \hat{f}_{n,b}(x) - \EE_{\widetilde{f}_0}[\hat{f}_{n,b}(x)] = \frac{1}{n} \sum_{i=1}^n Y_i(x),
\]
with iid copies $Y_i(x)$ of $Y_1(x)$.

\medskip
\emph{Lower bound on $\sigma_x^2$.}
Introduce the Stirling ratio
\begin{equation}\label{eq:Stirling.ratio}
R(u) \leqdef \frac{\sqrt{2\pi} \, e^{-u} u^{u + 1/2}}{\Gamma(u + 1)}, \quad u \geq 0.
\end{equation}
Substituting $\Gamma(u + 1) = \sqrt{2\pi} e^{-u}u^{u + 1/2}/R(u)$ into~\eqref{eq:Bb.def} with $u = x/b$ and $u = 2x/b$ yields
\begin{equation}\label{eq:Bb-R}
\int_0^{\infty} K_b(x,t)^2 \rd t \equiv B_b(x) = \frac{1}{2\sqrt{\pi}} \, b^{-1/2} x^{-1/2} \frac{R(x/b)^2}{R(2x/b)}.
\end{equation}
Since $R$ is increasing and $R(u) \leq 1$ for $u \geq 1$, for $b \leq 1/4$ we have $x/b \geq 1$ on $x\in [1/4,1/2]$ and thus $R(x/b)^2 / R(2x/b) \geq R(1)^2$. Hence, for $b \leq 1/4$ and all $x\in [1/4,1/2]$,
\[
B_b(x) \geq c \, b^{-1/2}
\]
for some constant $c > 0$.

Next, since $t\mapsto K_b(x,t)$ is decreasing on $[x,\infty)$ when $x \leq 1/2$, we have
\[
\int_{7/8}^{\infty} K_b(x,t)^2 \rd t \leq K_b(x,7/8) \int_{7/8}^{\infty} K_b(x,t) \rd t = K_b(x,7/8)\PP(\xi_x > 7/8).
\]
Both factors are $\exp(-c/b)$ uniformly for $x\in [1/4,1/2]$ by a Chernoff bound analogous to \eqref{eq:tail.bound.xi.x} (with threshold $7/8$), hence $\int_{7/8}^{\infty} K_b(x,t)^2 \rd t \leq \exp(-c/b)$. Therefore, for $b$ small enough,
\[
\int_0^{7/8} K_b(x,t)^2 \rd t
= B_b(x) - \int_{7/8}^{\infty} K_b(x,t)^2 \rd t
\geq \frac{1}{2} B_b(x) \geq c \, b^{-1/2}.
\]
Since $\widetilde{f}_0(t)=f_0(t)=1$ for $t\in[0,7/8]$, it follows that, for $b$ small enough,
\[
\EE_{\widetilde{f}_0}[\zeta_1(x)^2] = \int_0^1 K_b(x,t)^2\widetilde{f}_0(t) \, \rd t
\geq \int_0^{7/8} K_b(x,t)^2 \rd t \geq c \, b^{-1/2}.
\]
Moreover, since $\widetilde{f}_0\in\Sigma(\beta,\widetilde{L})$, we have $\|\widetilde{f}_0\|_{\infty}\leq \widetilde{L}$ and thus
\[
\big\{\EE_{\widetilde{f}_0}[\zeta_1(x)]\big\}^2 \leq \|\widetilde{f}_0\|_{\infty}^2 \leq \widetilde{L}^2,
\]
which is negligible compared to $b^{-1/2}$ as $b\to 0$. Hence, for $b$ small enough,
\begin{equation}\label{eq:sigma-lower}
\sigma_x^2 \equiv \Var_{\widetilde{f}_0}\{\zeta_1(x)\} \geq c \, b^{-1/2}, \quad x\in [1/4,1/2].
\end{equation}

\medskip
\emph{A Paley-Zygmund lower bound for $\EE_{\widetilde{f}_0}(|Z_n(x)|)$.}
Let
\[
M_x \leqdef \|Y_1(x)\|_{\infty} \leq 2 \sup_{t\in [0,1]} K_b(x,t) = 2K_b(x,x).
\]
Then $K_b(x,x) \leq C \, b^{-1/2}$ uniformly over $x\in [1/4,1/2]$ by \eqref{eq:sup-local}, so $M_x^2 \leq C \, b^{-1}$.

For the fourth moment, using independence and centering,
\[
\EE_{\widetilde{f}_0}[Z_n(x)^4]
= \frac{1}{n^4}\EE_{\widetilde{f}_0}\left[\left(\sum_{i=1}^n Y_i(x)\right)^4\right]
= \frac{1}{n^4}\left\{n \EE_{\widetilde{f}_0}\big[Y_1(x)^4\big] + 3n(n-1)\sigma_x^4\right\}.
\]
Moreover, $Y_1(x)^4 \leq M_x^2 Y_1(x)^2$, hence $\EE_{\widetilde{f}_0}[Y_1(x)^4] \leq M_x^2\sigma_x^2$. Therefore,
\[
\EE_{\widetilde{f}_0}[Z_n(x)^4] \leq \frac{M_x^2\sigma_x^2}{n^3} + \frac{3\sigma_x^4}{n^2}.
\]
Let $W = Z_n(x)^2$. Then $\EE_{\widetilde{f}_0}[W] = \EE_{\widetilde{f}_0}[Z_n(x)^2] = \sigma_x^2/n$ and $\EE_{\widetilde{f}_0}[W^2] = \EE_{\widetilde{f}_0}[Z_n(x)^4]$. By Paley-Zygmund with $\theta = 1/2$,
\begin{equation}\label{eq:to.bound.below}
\PP_{\widetilde{f}_0}\left(W \geq \frac{1}{2}\EE_{\widetilde{f}_0}[W]\right)
\geq \frac{(1 - \theta)^2(\EE_{\widetilde{f}_0}[W])^2}{\EE_{\widetilde{f}_0}[W^2]}
\geq \frac{c}{3 + M_x^2/(n\sigma_x^2)}.
\end{equation}
Using $M_x^2/(n\sigma_x^2) \leq C/(n b^{1/2})$ (from $M_x^2 \leq Cb^{-1}$ and~\eqref{eq:sigma-lower}), we find that if $n b^{1/2} \geq c_0$ then the above probability in \eqref{eq:to.bound.below} is bounded below by a positive constant. Hence, for $n b^{1/2} \geq c_0$,
\[
\EE_{\widetilde{f}_0}(|Z_n(x)|)
\geq \sqrt{\frac{1}{2}\EE_{\widetilde{f}_0}[Z_n(x)^2]} \, \PP_{\widetilde{f}_0}\left(|Z_n(x)| \geq \sqrt{\tfrac{1}{2} \EE_{\widetilde{f}_0}[Z_n(x)^2]}\right)
\geq c \, \frac{\sigma_x}{\sqrt{n}}
\geq c \, n^{-1/2}b^{-1/4},
\]
uniformly over $x\in [1/4,1/2]$, by~\eqref{eq:sigma-lower}. Integrating in $x$ over an interval of length $1/4$ proves the first inequality in $(b)$.

Finally, for $x\in [1/4,1/2]$, one has $\widetilde{f}_0(x)=f_0(x)=1$ and
\[
\EE_{\widetilde{f}_0}[\hat{f}_{n,b}(x)] = \int_0^1 K_b(x,t)\widetilde{f}_0(t) \, \rd t
= \int_0^1 K_b(x,t) \, \rd t + \int_{7/8}^1 K_b(x,t)\{\widetilde{f}_0(t)-1\} \, \rd t.
\]
Therefore,
\[
\big|\EE_{\widetilde{f}_0}[\hat{f}_{n,b}(x)] - \widetilde{f}_0(x)\big|
\leq \PP(\xi_x>1) + \|\widetilde{f}_0-1\|_{\infty} \, \PP(\xi_x>7/8)
\leq \exp(-c_0/b),
\]
for some $c_0>0$ and all $b$ small enough, using Chernoff bounds as above. This yields the last inequality in $(b)$.
\end{proof}

We now give the proof of Proposition~\ref{prop:2}. Fix $p\in [1,\infty)$, $\beta\in (2,\infty)$ and $\widetilde{L} > 1$. Let $(b_n)_{n\in \N}\subseteq (0,1)$ be arbitrary and set $b = b_n$.

By~\eqref{eq:Lp-to-L1} and the fact that $\{\widetilde{f}_0,\widetilde{f}_3\}\subseteq \Sigma(\beta,\widetilde{L})$, we have
\[
\begin{aligned}
R_n\{\hat{f}_{n,b},\Sigma(\beta,\widetilde{L})\}
&\geq \frac{1}{2}\left(R_n(\hat{f}_{n,b},\widetilde{f}_0) + R_n(\hat{f}_{n,b},\widetilde{f}_3)\right) \\
&\geq \frac{1}{2}\left(\EE_{\widetilde{f}_0}(\|\hat{f}_{n,b} - \widetilde{f}_0\|_{L^1([0,1])}) + \EE_{\widetilde{f}_3}(\|\hat{f}_{n,b} - \widetilde{f}_3\|_{L^1([0,1])})\right).
\end{aligned}
\]
Moreover, since $\|\cdot\|_{L^1([0,1])}$ is convex, Jensen's inequality yields
\[
\EE_{\widetilde{f}_3}(\|\hat{f}_{n,b} - \widetilde{f}_3\|_{L^1([0,1])}) \geq \|\EE_{\widetilde{f}_3}[\hat{f}_{n,b}] - \widetilde{f}_3\|_{L^1([0,1])},
\]
so
\begin{equation}\label{eq:sup-lower}
R_n\{\hat{f}_{n,b},\Sigma(\beta,\widetilde{L})\} \geq \frac{1}{2}\left(\EE_{\widetilde{f}_0}(\|\hat{f}_{n,b} - \widetilde{f}_0\|_{L^1([0,1])}) + \|\EE_{\widetilde{f}_3}[\hat{f}_{n,b}] - \widetilde{f}_3\|_{L^1([0,1])}\right).
\end{equation}

To handle the $\liminf$ in Proposition~\ref{prop:2}, set
\[
a_n \leqdef \frac{R_n\{\hat{f}_{n,b_n},\Sigma(\beta,\widetilde{L})\}}{r_n\{\Sigma(\beta,\widetilde{L})\}}.
\]
Let $(n_k)_{k\in\N}$ be a subsequence such that $a_{n_k}\to \liminf_{n\to\infty} a_n$. It suffices to prove that $a_{n_k}\to +\infty$. For notational simplicity, we relabel the subsequence and write $n$ for $n_k$ and $b_n$ for $b_{n_k}$ below.

We now distinguish three cases.

\medskip
\emph{Case 1: $b_n$ does not tend to $0$}
Passing to a further subsequence, we may assume $b_n\to b_\star\in (0,1]$. Then
\[
\begin{aligned}
\|\EE_{\widetilde{f}_3}[\hat{f}_{n,b_n}] - \widetilde{f}_3\|_{L^1([0,1])}
&= \left\|\int_0^1 K_{b_n}(\cdot,t)\widetilde{f}_3(t) \, \rd t - \widetilde{f}_3\right\|_{L^1([0,1])} \\
&\longrightarrow \left\|\int_0^1 K_{b_\star}(\cdot,t)\widetilde{f}_3(t) \, \rd t - \widetilde{f}_3\right\|_{L^1([0,1])},
\end{aligned}
\]
by dominated convergence. Since $b_\star>0$, the limit above is strictly positive. Hence the right-hand side of~\eqref{eq:sup-lower} is bounded away from $0$ along this subsequence, and since $r_n\{\Sigma(\beta,\widetilde{L})\}\to 0$, we obtain $a_n\to +\infty$.

\medskip
\emph{Case 2: $b_n\to 0$ but $n b_n^{1/2}$ does not tend to $+\infty$.}
Passing to a further subsequence, we may assume $s_n \leqdef n b_n^{1/2}\leq M$ for all $n$ and some $M<\infty$. Fix $x\in [1/4,1/2]$ and set
\[
\delta_n \leqdef \sqrt{C_0 \ln(1/b_n)},
\]
where $C_0>0$ is a large numerical constant to be chosen below. Consider the event
\[
A_n(x) \leqdef \bigcap_{i=1}^n \{|X_i-x|>\delta_n b_n^{1/2}\}.
\]
For $n$ large enough, the interval $\{u:|u-x|\leq \delta_n b_n^{1/2}\}$ is contained in $[0,7/8]$, where $\widetilde{f}_0\equiv 1$, so
\[
\PP_{\widetilde{f}_0}\{A_n(x)\} = \big(1-2\delta_n b_n^{1/2}\big)^n \geq \exp\{-4 \delta_n n b_n^{1/2}\} = \exp\{-4\delta_n s_n\}.
\]
On $A_n(x)$, since $u\mapsto K_{b_n}(x,u)$ is unimodal with mode at $u=x$ (increasing on $[0,x]$ and decreasing on $[x,\infty)$), we have
\[
\hat{f}_{n,b_n}(x)\leq \sup_{\{|u-x|\geq \delta_n b_n^{1/2}\}}K_{b_n}(x,u)
= \max\Big\{K_{b_n}(x,x-\delta_n b_n^{1/2}),\,K_{b_n}(x,x+\delta_n b_n^{1/2})\Big\}.
\]
Write, for any $\delta\in\R$ such that $x+\delta b^{1/2}>0$,
\[
K_b(x,x + \delta b^{1/2}) = K_b(x,x) Q_{b,\delta}(x), \quad
Q_{b,\delta}(x) \leqdef \exp\left\{\frac{x}{b}\ln\left(1 + \frac{\delta b^{1/2}}{x}\right)-\frac{\delta}{b^{1/2}}\right\},
\]
as in the proof of Lemma~\ref{lem:K-local-lower} in Appendix~\ref{app}. Since $\delta_n b_n^{1/2}\to 0$ and $\delta_n^3 \hspace{0.2mm} b_n^{1/2}\to 0$ under $b_n\to 0$, the same Taylor expansion yields $\ln\{ Q_{b_n,\pm \delta_n}(x)\} \leq -c \, \delta_n^2$ for all $n$ large enough and all $x\in[1/4,1/2]$. Using $K_{b_n}(x,x)\leq C b_n^{-1/2}$ (see \eqref{eq:sup-local} with $x\in[1/4,1/2]$), we obtain for all $n$ large enough,
\begin{equation}\label{eq:last.display}
\max\Big\{K_{b_n}(x,x-\delta_n b_n^{1/2}),\,K_{b_n}(x,x+\delta_n b_n^{1/2})\Big\}
\leq C b_n^{-1/2}\exp(-c \, \delta_n^2)
\leq C \, b_n^{c \, C_0 - 1/2}.
\end{equation}
By choosing $C_0$ large enough so that $c \, C_0 - 1/2>0$, the last display \eqref{eq:last.display} is at most $1/2$ for all $n$ large enough. Since $\widetilde{f}_0(x)=1$ on $[1/4,1/2]$, it follows that, for all $n$ large enough,
\[
\EE_{\widetilde{f}_0}\big[|\hat{f}_{n,b_n}(x)-\widetilde{f}_0(x)|\big]
\geq \frac{1}{2} \, \PP_{\widetilde{f}_0}\{A_n(x)\}
\geq \frac{1}{2} \, \exp\{-4 \delta_n s_n\}.
\]
Integrating over $x\in[1/4,1/2]$ and using~\eqref{eq:sup-lower} yields
\[
R_n\{\hat{f}_{n,b_n},\Sigma(\beta,\widetilde{L})\} \geq c \exp\{-C \, \delta_n s_n\}.
\]
Since $b_n=s_n^2/n^2$, we have $\delta_n=\sqrt{C_0\ln(1/b_n)}=\sqrt{2C_0\ln(n/s_n)}$, so
\[
R_n\{\hat{f}_{n,b_n},\Sigma(\beta,\widetilde{L})\} \geq c \, \exp\{-\widetilde{C} s_n \sqrt{\ln(n/s_n)}\}.
\]
Since $s\mapsto s\sqrt{\ln(n/s)}$ is increasing on $(0,n/\sqrt{e})$, and $s_n \leq M < n/\sqrt{e}$ for $n$ large enough, using the minimax rate in~\eqref{eq:minimax.rate} yields
\[
\frac{R_n\{\hat{f}_{n,b_n},\Sigma(\beta,\widetilde{L})\}}{r_n\{\Sigma(\beta,\widetilde{L})\}}
\geq c \, n^{\beta/(2\beta + 1)} \exp\{-\widetilde{C} M \sqrt{\ln(n/M)}\} \longrightarrow +\infty.
\]
Hence $a_n\to +\infty$ also in this subcase.

\medskip
\emph{Case 3: $b_n\to 0$ and $n b_n^{1/2}\to + \infty$.}
Then for all $n$ large enough, $b_n\in (0,b_0]$ and $n b_n^{1/2} \geq c_0$, so Lemma~\ref{lem:prop2-key} applies. From~\eqref{eq:sup-lower} and Lemma~\ref{lem:prop2-key}(a) and~(b),
\[
R_n\{\hat{f}_{n,b_n},\Sigma(\beta,\widetilde{L})\} \geq c \, (b_n + n^{-1/2} b_n^{-1/4}) - \exp(-c/b_n).
\]
Since $\exp(-c/b_n) = o(b_n)$ as $b_n\to 0$, for $n$ large enough, we get
\[
R_n\{\hat{f}_{n,b_n},\Sigma(\beta,\widetilde{L})\} \geq c \, (b_n + n^{-1/2}b_n^{-1/4}).
\]
For every $n$ and every $b > 0$, writing $u = b^{1/4}$ gives
\[
b + n^{-1/2}b^{-1/4} = u^4 + \frac{1}{\sqrt{n} u} = n^{-2/5}\left((u n^{1/10})^4 + (u n^{1/10})^{-1}\right) \geq c \, n^{-2/5},
\]
because $\inf_{t > 0}(t^4 + t^{-1}) > 0$. Hence,
\begin{equation}\label{eq:rate-lower-25}
R_n\{\hat{f}_{n,b_n},\Sigma(\beta,\widetilde{L})\} \geq c \, n^{-2/5}.
\end{equation}

Since $\beta > 2$ implies $\beta/(2\beta + 1) > 2/5$, we have $n^{-2/5}/n^{-\beta/(2\beta + 1)}\to \infty$. Combining with~\eqref{eq:rate-lower-25} yields
\[
\liminf_{n\to \infty}\frac{R_n\{\hat{f}_{n,b_n},\Sigma(\beta,\widetilde{L})\}}{r_n\{\Sigma(\beta,\widetilde{L})\}}
\geq \liminf_{n\to \infty} c \, n^{\beta/(2\beta + 1)-2/5}
= +\infty.
\]
This completes the proof of Proposition~\ref{prop:2}.

\subsection{Proof of Proposition~\ref{prop:3}}\label{sec:proof.prop.3}

First, we state two auxiliary technical lemmas, whose proofs are deferred to Appendix~\ref{app}.

\begin{lemma}\label{lem:nonminimax-var}
Fix $p\in [2,\infty)$. Let $\widetilde{f}_0$ be as in the proof of Proposition~\ref{prop:2}. There exist $b_0\in (0,1)$ and $c > 0$ such that for all $b\in (0,b_0]$ and all $n\in \N$,
\[
\EE_{\widetilde{f}_0}\big(\|\hat{f}_{n,b} - \widetilde{f}_0\|_p^p\big) \geq c \, \frac{\mathcal{I}(b,p)}{(n b^{1/2})^{p/2}}, \quad \mathcal{I}(b,p) \leqdef \int_b^{1/2} x^{-p/4} \rd x.
\]
Consequently,
\[
R_n(\hat{f}_{n,b},\widetilde{f}_0) \geq c \, \frac{\{\mathcal{I}(b,p)\}^{1/p}}{(n b^{1/2})^{1/2}}.
\]
\end{lemma}

\begin{lemma}\label{lem:nonminimax-bias}
Fix $\beta\in (0,2]$. There exist $\widetilde{L} > 1$, $b_1\in (0,1)$, and $c > 0$ such that, for every $b\in (0,b_1]$, there exists a density $\widetilde{f}_{\beta,b}\in \Sigma(\beta,\widetilde{L})$ satisfying
\[
\|\EE_{\widetilde{f}_{\beta,b}}[\hat{f}_{n,b}] - \widetilde{f}_{\beta,b}\|_{L^1([0,1])} \geq c \, b^{\beta/2},
\]
and therefore, for every $p\in [1,\infty)$,
\[
R_n(\hat{f}_{n,b},\widetilde{f}_{\beta,b}) \geq c \, b^{\beta/2}.
\]
\end{lemma}

Now, we find a general lower bound on the maximal risk. Set $b_{\ast}\leqdef b_0\wedge b_1\wedge 1/4$. Fix $n\in \N$ and $b\in (0,b_{\ast}]$. Since $\widetilde{f}_0\in \Sigma(\beta,\widetilde{L})$ and $\widetilde{f}_{\beta,b}\in \Sigma(\beta,\widetilde{L})$ (for the explicit definition, see \eqref{eq:fbeta-def} together with the smoothing described in the proof of Lemma~\ref{lem:nonminimax-bias}), one has
\begin{equation}\label{eq:sup-lower-prop3}
R_n\{\hat{f}_{n,b},\Sigma(\beta,\widetilde{L})\} \geq \frac{1}{2}\left\{R_n(\hat{f}_{n,b},\widetilde{f}_0) + R_n(\hat{f}_{n,b},\widetilde{f}_{\beta,b})\right\}.
\end{equation}
Applying Lemmas~\ref{lem:nonminimax-var} and~\ref{lem:nonminimax-bias} to~\eqref{eq:sup-lower-prop3} yields that for all $b \leq b_{\ast}$,
\begin{equation}\label{eq:key-lower}
R_n\{\hat{f}_{n,b},\Sigma(\beta,\widetilde{L})\} \geq c \left\{\frac{\{\mathcal{I}(b,p)\}^{1/p}}{(n b^{1/2})^{1/2}} + b^{\beta/2}\right\}.
\end{equation}

Next, we optimize the lower bound we just found. Fix $p\in [4,\infty)$ and consider the function
\[
\Phi_n(b) \leqdef \frac{\{\mathcal{I}(b,p)\}^{1/p}}{(n b^{1/2})^{1/2}} + b^{\beta/2}, \quad b\in (0,b_{\ast}].
\]
Then~\eqref{eq:key-lower} implies $R_n\{\hat{f}_{n,b},\Sigma(\beta,\widetilde{L})\} \geq c \, \Phi_n(b)$ for all $b\in (0,b_{\ast}]$. We next extend this reduction to arbitrary bandwidth sequences. Since $\widetilde{f}_0(x)=1$ for all $x\in[0,7/8]$, Jensen's inequality gives, for $b\in[b_{\ast},1)$,
\[
R_n\{\hat{f}_{n,b},\Sigma(\beta,\widetilde{L})\} \geq R_n(\hat{f}_{n,b},\widetilde{f}_0) \geq \|\EE_{\widetilde{f}_0}[\hat{f}_{n,b}] - \widetilde{f}_0\|_p.
\]
For $x\in[3/2,2]$, we have $\widetilde{f}_0(x)=0$, and therefore
\[
\EE_{\widetilde{f}_0}[\hat{f}_{n,b}(x)] - \widetilde{f}_0(x) = \int_0^1 K_b(x,t)\widetilde{f}_0(t)\rd t \geq \int_{1/4}^{1/2} K_b(x,t)\rd t \geq c,
\]
where the last inequality is uniform over $x\in[3/2,2]$ and $b\in[b_{\ast},1]$ by continuity and positivity of $K_b(x,t)$ on the compact set $[3/2,2]\times[1/4,1/2]\times[b_{\ast},1]$. Hence
\[
R_n\{\hat{f}_{n,b},\Sigma(\beta,\widetilde{L})\} \geq c, \quad b\in[b_{\ast},1).
\]
On the other hand, $\inf_{b\in(0,b_{\ast}]}\Phi_n(b)\to 0$ as $n\to\infty$, for instance by taking $b=n^{-1/2}$. Thus, after decreasing $c$ and taking $n$ large enough, the preceding two lower bounds imply that, for every $b\in(0,1)$,
\[
R_n\{\hat{f}_{n,b},\Sigma(\beta,\widetilde{L})\} \geq c \inf_{u\in (0,b_{\ast}]}\Phi_n(u).
\]
Consequently, for any bandwidth sequence $(b_n)$ with $b_n\in(0,1)$, eventually,
\begin{equation}\label{eq:inf-bound}
R_n\{\hat{f}_{n,b_n},\Sigma(\beta,\widetilde{L})\} \geq c \inf_{b\in (0,b_{\ast}]}\Phi_n(b).
\end{equation}

\medskip
\emph{Case $p = 4$.}
Here $\mathcal{I}(b,4) = \int_b^{1/2} x^{-1} \rd x = \ln(1/(2 b))$, so for $b$ small enough,
\[
\Phi_n(b) \geq c\left\{\frac{|\ln b|^{1/4}}{(n b^{1/2})^{1/2}} + b^{\beta/2}\right\} = c\left\{n^{-1/2}b^{-1/4}|\ln b|^{1/4} + b^{\beta/2}\right\}.
\]
Set
\[
b_n^{\star} \leqdef n^{-2/(2\beta + 1)}(\ln n)^{1/(2\beta + 1)}.
\]
For $n$ large enough, one has $b_n^{\star}\in (0,b_{\ast}]$. If $b \geq b_n^{\star}$, then
\[
\Phi_n(b) \geq c \, b^{\beta/2} \geq c \, (b_n^{\star})^{\beta/2} = c \, n^{-\beta/(2\beta + 1)}(\ln n)^{\beta/(2(2\beta + 1))}.
\]
If $b < b_n^{\star}$, then $b^{-1/4} \geq (b_n^{\star})^{-1/4} = n^{1/(2(2\beta + 1))}(\ln n)^{-1/(4(2\beta + 1))}$ and $|\ln b| \geq |\ln b_n^{\star}| \geq c\ln n$ for large $n$, hence
\[
n^{-1/2}b^{-1/4}|\ln b|^{1/4}
\geq c \, n^{-1/2}n^{1/(2(2\beta + 1))}(\ln n)^{-1/(4(2\beta + 1))}(\ln n)^{1/4}
= c \, n^{-\beta/(2\beta + 1)}(\ln n)^{\beta/(2(2\beta + 1))}.
\]
Therefore,
\[
\inf_{b\in (0,b_{\ast}]}\Phi_n(b) \geq c \, n^{-\beta/(2\beta + 1)}(\ln n)^{\beta/(2(2\beta + 1))}.
\]
Combining this with~\eqref{eq:inf-bound} and the minimax rate stated in \eqref{eq:minimax.rate} yields
\[
\frac{R_n\{\hat{f}_{n,b_n},\Sigma(\beta,\widetilde{L})\}}{r_n\{\Sigma(\beta,\widetilde{L})\}}
\geq c \, (\ln n)^{\beta/(2(2\beta + 1))} \longrightarrow + \infty,
\]
which proves the claim when $p = 4$.

\medskip
\emph{Case $p > 4$.}
For $b\in (0,b_{\ast}]$,
\[
\mathcal{I}(b,p) = \int_b^{1/2} x^{-p/4} \rd x \geq \int_b^{2 b} x^{-p/4} \rd x = c \, b^{1 - p/4},
\]
so $\{\mathcal{I}(b,p)\}^{1/p} \geq c \, b^{1/p-1/4}$. Hence, for $b \leq b_{\ast}$,
\[
\Phi_n(b) \geq c \, (n^{-1/2}b^{-1/4}b^{1/p-1/4} + b^{\beta/2}) = c \, (n^{-1/2}b^{-1/2 + 1/p} + b^{\beta/2}).
\]
Let $\alpha = \beta/2$ and $\gamma = 1/2-1/p > 0$. Then the right-hand side has the form $c \, (n^{-1/2} b^{-\gamma} + b^{\alpha})$. The minimum is attained at a bandwidth $b$ of order $n^{-1/(2(\alpha + \gamma))}$, so that
\[
\inf_{b\in (0,b_{\ast}]} c \, (n^{-1/2} b^{-\gamma} + b^{\alpha}) \geq c \, n^{-\alpha/(2(\alpha + \gamma))} = c \, n^{-\beta/(2\beta + 2-4/p)}.
\]
Therefore, for $n$ large enough,
\[
\inf_{b\in (0,b_{\ast}]}\Phi_n(b) \geq c \, n^{-\beta/(2\beta + 2-4/p)}.
\]
Since $p > 4$ implies $2\beta + 2-4/p > 2\beta + 1$, one has $n^{\beta/(2\beta + 2-4/p)} = \oo(n^{\beta/(2\beta + 1)})$ and thus, by~\eqref{eq:minimax.rate},
\[
\frac{R_n\{\hat{f}_{n,b_n},\Sigma(\beta,\widetilde{L})\}}{r_n\{\Sigma(\beta,\widetilde{L})\}} \geq c \, \frac{n^{\beta/(2\beta + 1)}}{n^{\beta/(2\beta + 2-4/p)}} \longrightarrow +\infty.
\]
This completes the proof of Proposition~\ref{prop:3}.

\begin{appendices}

\section{Proofs of technical lemmas}\label{app}

\begin{proof}[Proof of Lemma~\ref{lem:nonminimax-var}]
Fix $p\in [2,\infty)$ and let $\widetilde{f}_0$ be as in the proof of Proposition~\ref{prop:2}. Recall in particular that $\widetilde{f}_0(x)=1$ for all $x\in[0,7/8]$, and hence for all $x\in[b,1/2]$ when $b\leq 1/2$.
We use the inequality $\EE[|U|^p] \geq 2^{-p} \, \EE[|U - \EE[U]|^p]$, which holds for any random variable $U$ and $p \geq 1$. Applying this pointwise to $\smash{U = \hat{f}_{n,b}(x) - \widetilde{f}_0(x)}$ yields
\[
\EE_{\widetilde{f}_0}\big[|\hat{f}_{n,b}(x) - \widetilde{f}_0(x)|^p\big] \geq 2^{-p} \, \EE_{\widetilde{f}_0}\big[|\hat{f}_{n,b}(x) - \EE_{\widetilde{f}_0}[\hat{f}_{n,b}(x)]|^p\big].
\]
Since $p \geq 2$, integrating with respect to $x$ and applying Jensen's inequality yield
\begin{equation}\label{eq:var-lower-int}
\EE_{\widetilde{f}_0}\big(\|\hat{f}_{n,b} - \widetilde{f}_0\|_p^p\big) \geq 2^{-p} \int_0^{\infty} \big\{\Var_{\widetilde{f}_0}(\hat{f}_{n,b}(x))\big\}^{p/2} \rd x.
\end{equation}

For each fixed $x\geq 0$,
\[
\Var_{\widetilde{f}_0}(\hat{f}_{n,b}(x)) = \frac{1}{n} \Var_{\widetilde{f}_0}\big(K_b(x,X_1)\big) = \frac{1}{n} \left(\EE_{\widetilde{f}_0}[K_b(x,X_1)^2] - \{\EE_{\widetilde{f}_0}[K_b(x,X_1)]\}^2\right).
\]
Recall the expansion of the squared kernel integral \eqref{eq:Bb-R} derived in the proof of Lemma~\ref{lem:prop2-key}:
\[
\int_0^{\infty} K_b(x,t)^2 \rd t = \frac{1}{2\sqrt{\pi}} \, b^{-1/2} x^{-1/2} \frac{R(x/b)^2}{R(2x/b)},
\]
where $R$ is the Stirling ratio function in \eqref{eq:Stirling.ratio}. For $x\in [b, 1/2]$, we have $x/b \geq 1$. Since $R$ is increasing and positive, the ratio $R(x/b)^2/R(2x/b)$ is bounded from below by a positive constant. Furthermore, as argued in Lemma~\ref{lem:prop2-key}, the integral on the tail $[7/8,\infty)$ is exponentially small.
Thus, for $b$ small enough and all $x\in [b,1/2]$,
\[
\EE_{\widetilde{f}_0}[K_b(x,X_1)^2]
= \int_0^1 K_b(x,t)^2 \widetilde{f}_0(t) \, \rd t
\geq \int_0^{7/8} K_b(x,t)^2 \rd t
\geq c \, b^{-1/2} x^{-1/2}.
\]
Moreover, $\{\EE_{\widetilde{f}_0}[K_b(x,X_1)]\}^2 \leq \|\widetilde{f}_0\|_{\infty}^2 < \infty$ is negligible compared to the second moment above as $b\to 0$. It follows (possibly shrinking $b_0$) that
\[
\Var_{\widetilde{f}_0}(\hat{f}_{n,b}(x)) \geq \frac{c}{n} b^{-1/2} x^{-1/2}, \quad x\in [b,1/2], ~b\in (0,b_0].
\]

Plugging this into~\eqref{eq:var-lower-int} and restricting the integral to $[b,1/2]$ yields
\[
\EE_{\widetilde{f}_0}(\|\hat{f}_{n,b} - \widetilde{f}_0\|_p^p)
\geq c \int_b^{1/2}\left(\frac{b^{-1/2} x^{-1/2}}{n}\right)^{p/2} \rd x
= c \, \frac{1}{(n b^{1/2})^{p/2}} \int_b^{1/2} x^{-p/4} \rd x,
\]
which is the first claim. Taking the $p$th root gives the second.
\end{proof}

Before proving Lemma~\ref{lem:nonminimax-bias}, we need the following local lower bound on the gamma kernel.

\begin{lemma}\label{lem:K-local-lower}
Fix $0 < a_0 < a_1 < 1$ and $\delta\in (0,3)$. Then there exist $b_2\in (0,1)$ and $c > 0$ such that for all $b\in (0,b_2]$ and all $x\in [a_0,a_1]$,
\[
K_b(x,x + \delta b^{1/2}) \geq c \, b^{-1/2}.
\]
\end{lemma}

\begin{proof}[Proof of Lemma~\ref{lem:K-local-lower}]
Write
\[
K_b(x,x + \delta b^{1/2}) = K_b(x,x) Q_{b,\delta}(x), \quad
Q_{b,\delta}(x) \leqdef \exp\left\{\frac{x}{b}\ln\left(1 + \frac{\delta b^{1/2}}{x}\right)-\frac{\delta}{b^{1/2}}\right\}.
\]
Using $\ln(1 + u) = u-u^2/2 + \OO(u^3)$ with $u = \delta b^{1/2}/x$ gives, uniformly for $x\in [a_0,a_1]$,
\[
\ln Q_{b,\delta}(x)
= \left(\frac{x}{b}\right)\left(\frac{\delta b^{1/2}}{x}-\frac{\delta^2 b}{2x^2} + \OO(b^{3/2})\right) -\frac{\delta}{b^{1/2}}
= -\frac{\delta^2}{2x} + \OO(b^{1/2}).
\]
Therefore, $Q_{b,\delta}(x)\to \exp\{-\delta^2/(2x)\}$ uniformly in $x\in [a_0,a_1]$ as $b\to 0$, so $Q_{b,\delta}(x) \geq c$ for $b$ small enough. Next, using the Stirling ratio function in \eqref{eq:Stirling.ratio}, Stirling's formula gives
\[
K_b(x,x) = \frac{x^{x/b} e^{-x/b}}{b^{x/b + 1} \Gamma(x/b + 1)} = \frac{R(x/b)}{\sqrt{2\pi x b}} \geq c \, b^{-1/2},
\]
uniformly on $x\in [a_0,a_1]$ and $b$ small enough, since $x/b\to \infty$ and $R(\cdot)$ is increasing with $\lim_{u\to \infty}R(u) = 1$.
Combining the bounds on $K_b(x,x)$ and $Q_{b,\delta}(x)$ yields the claim.
\end{proof}

\begin{proof}[Proof of Lemma~\ref{lem:nonminimax-bias}]
Fix $\beta\in (0,2]$ and $L > 1$. Let $\psi:\R\to \R$ be the compactly supported $C^2$ function
\[
\psi(u) \leqdef (1 - u^2)^3 \, \ind_{[-1,1]}(u),
\]
so that $\psi(0) = 1$, $0 \leq \psi \leq 1$, and $\psi$ is nonincreasing on $[0,1]$.
Set
\[
L_{\beta} \leqdef \frac{L}{16}, \quad N \leqdef \left\lceil \frac{1}{24 b^{1/2}}\right\rceil, \quad t^{(k)} \leqdef \frac14 + 3 b^{1/2}(2k-1), \quad k = 1,\ldots,2N,
\]
and define the preliminary (non-smooth at $x=1$) test density
\begin{equation}\label{eq:fbeta-def}
f_{\beta,b}(x) \leqdef \ind_{[0,1]}(x) + L_{\beta}(3 b^{1/2})^{\beta} \sum_{k=1}^{2N}(-1)^k \, \psi\left(\frac{x-t^{(k)}}{3 b^{1/2}}\right) \ind_{[0,1]}(x),
\end{equation}
which is illustrated in Figure~\ref{fig:fbeta}.

\begin{figure}[!ht]
\centering
\includegraphics[width=0.93\textwidth]{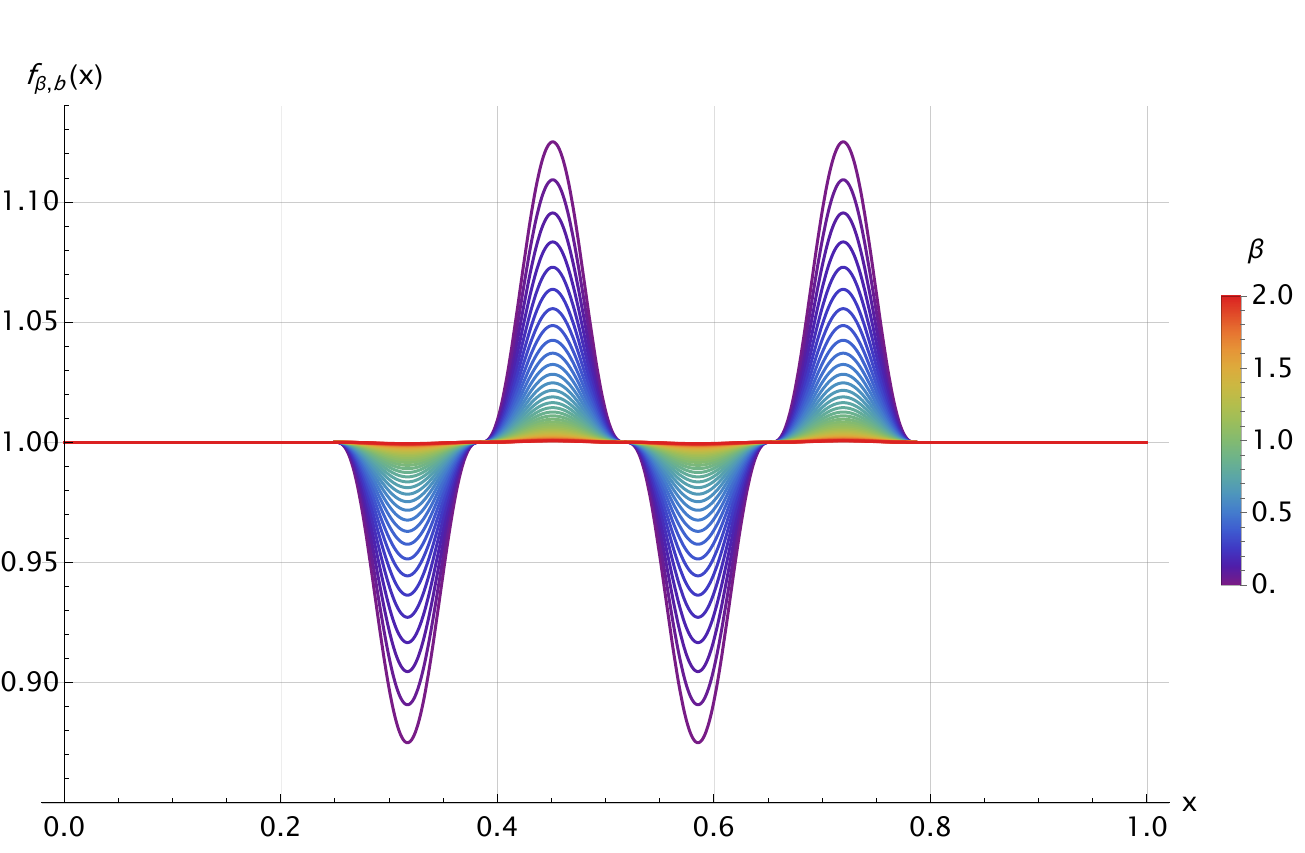}
\caption{Visualization of the preliminary (non-smooth at $x=1$) test density $f_{\beta,b}$ defined in \eqref{eq:fbeta-def} for $L=2$, $b=0.0005$, and various values of $\beta \in (0,2]$. The oscillations represent the localized perturbations used to lower bound the maximal risk.}
\label{fig:fbeta}
\end{figure}

For $b$ small enough, the supports of the translated bumps are disjoint and contained in $[0,7/8]$, and $L_{\beta}(3 b^{1/2})^{\beta} \leq 1/2$, so $f_{\beta,b} \geq 1/2$ on $[0,1]$. Moreover, since there are as many $+1$ as $-1$ signs in the sum, the integral of the bump sum is $0$, hence $\smash{\int_0^1 f_{\beta,b}(x) \rd x = 1}$. Thus $f_{\beta,b}$ is a density supported on $[0,1]$. However, due to the hard cutoff at $x=1$, it does not satisfy the smoothness requirements imposed by the definition of $\Sigma(\beta,L)$.

Let $\widetilde{f}_{\beta,b}$ denote a $C^2$ modification of $f_{\beta,b}$ on $[7/8,1]$ such that:
\begin{itemize}\setlength\itemsep{0em}
\item $\widetilde{f}_{\beta,b}=f_{\beta,b}$ on $[0,7/8]$;
\item $\widetilde{f}_{\beta,b}$ is supported on $[0,1]$ and $\int_0^1 \widetilde{f}_{\beta,b}(x) \, \rd x=1$;
\item $\widetilde{f}_{\beta,b}$ is $C^2$ on $(0,\infty)$, is $C^{\infty}$ in a neighbourhood of $1$, and $\widetilde{f}_{\beta,b}^{(k)}(1)=0$ for all $k\in\N_0$;
\item $\widetilde{f}_{\beta,b}\in \Sigma(\beta,\widetilde{L})$ for an appropriate $\widetilde{L} > 1$.
\end{itemize}
Again, we omit the explicit smoothing, since it does not affect the lower bounds below: all points $x$ and $u$ that appear in the argument lie in $[0,3/4 + 12 b^{1/2}]$, which itself is contained in say $[0,5/6]$ for $b$ small enough, and for such $x$ the kernel mass on $[7/8,1]$ is exponentially small in $1/b$ given that $7/8 - 5/6 > 0$. In particular, the bias lower bound obtained below for $f_{\beta,b}$ carries over to $\widetilde{f}_{\beta,b}$ (possibly after shrinking~$b_1$ and adjusting constants).

\medskip
\emph{Lower bound on the $L^1([0,1])$ bias.}
Let $\Delta_N \leqdef \{k\in \{1,\ldots,2N\}: k \text{ is even}\}$, so that the corresponding bumps are \emph{positive}. Fix $ \e\in (0,1/2]$ (to be chosen later) and define the intervals
\[
T_k(\e,b) \leqdef \left\{x: |x-t^{(k)}| \leq \e b^{1/2}\right\}, \quad
I_k(b) \leqdef \left\{u: b^{1/2} \leq |u-t^{(k)}| \leq 2 b^{1/2}\right\}.
\]
For $k\in \Delta_N$ and $x\in T_k(\e,b)$, only the $k$th bump is active in~\eqref{eq:fbeta-def}, so
\[
f_{\beta,b}(x) = 1 + L_{\beta}(3 b^{1/2})^{\beta} \, \psi\left(\frac{x-t^{(k)}}{3 b^{1/2}}\right),
\]
and similarly for $u\in I_k(b)$. Since $\psi$ is nonincreasing on $[0,1]$, for $x\in T_k(\e,b)$ (so that $|(x-t^{(k)})/(3 b^{1/2})| \leq \e/3$)
and $u\in I_k(b)$ (so that $|(u-t^{(k)})/(3 b^{1/2})|\in [1/3,2/3]$), one has
\begin{equation}\label{eq:f.beta.diff.lower.bound}
f_{\beta,b}(x) - f_{\beta,b}(u)
= L_{\beta}(3 b^{1/2})^{\beta}\left\{\psi\left(\frac{x-t^{(k)}}{3 b^{1/2}}\right) - \psi\left(\frac{u-t^{(k)}}{3 b^{1/2}}\right)\right\}
\geq c \, b^{\beta/2},
\end{equation}
where $c \leqdef L_{\beta} 3^{\beta}\{\psi(\e/3) - \psi(1/3)\} > 0$ for fixed $ \e$.

Now fix $k\in \Delta_N$ and $x\in T_k(\e,b)$. Since $f_{\beta,b}$ is supported on $[0,1]$ and $\int_0^{\infty} K_b(x,u)\rd u = 1$, we have
\[
\EE_{f_{\beta,b}}[\hat{f}_{n,b}(x)] - f_{\beta,b}(x) = \int_0^{\infty} K_b(x,u)\{f_{\beta,b}(u) - f_{\beta,b}(x)\} \rd u,
\]
so we can write
\[
\Big|\EE_{f_{\beta,b}}[\hat{f}_{n,b}(x)] - f_{\beta,b}(x)\Big| \geq U_k(x) - V(x),
\]
where
\[
\begin{aligned}
U_k(x) &\leqdef \int_{I_k(b)} K_b(x,u)\{f_{\beta,b}(x) - f_{\beta,b}(u)\} \rd u, \\
V(x) &\leqdef \int_{\{f_{\beta,b}(u) \geq f_{\beta,b}(x)\}} K_b(x,u)\{f_{\beta,b}(u) - f_{\beta,b}(x)\} \rd u.
\end{aligned}
\]
From the previous pointwise difference lower bound \eqref{eq:f.beta.diff.lower.bound},
\begin{equation}\label{eq:Ak-lower}
U_k(x) \geq c \, b^{\beta/2} \int_{I_k(b)} K_b(x,u) \rd u.
\end{equation}
To bound the integral in \eqref{eq:Ak-lower} from below, note that for $x\in T_k(\e,b)$ and $u\in [t^{(k)} + \frac{3}{2} b^{1/2}, t^{(k)} + 2 b^{1/2}]$, one has $u\in I_k(b)$ and
\[
u - x\in \big[(\tfrac{3}{2}-\e)b^{1/2}, \, (2+\e)b^{1/2}\big] \subseteq [b^{1/2}, (5/2)b^{1/2}],
\]
provided $ \e \leq 1/2$. Moreover, for $b$ small enough, $x\in [1/4,3/4 + 12 b^{1/2}] \subseteq [0,5/6]$. Since for each fixed $x>0$ the function $u\mapsto K_b(x,u)$ is decreasing on $[x,\infty)$ (its mode is at $u=x$), it follows that, for all such $u$,
\[
K_b(x,u) \geq K_b\left(x,x+\frac{5}{2}b^{1/2}\right).
\]
Hence Lemma~\ref{lem:K-local-lower} (with $a_0 = 1/4$, $a_1 = 5/6$ and $\delta = 5/2$) implies
\[
\inf_{u\in [t^{(k)} + \frac{3}{2} b^{1/2}, \, t^{(k)} + 2 b^{1/2}]} K_b(x,u) \geq c \, b^{-1/2},
\]
so
\begin{equation}\label{eq:intK-lower}
\int_{I_k(b)} K_b(x,u) \rd u \geq \int_{t^{(k)} + \frac{3}{2} b^{1/2}}^{t^{(k)} + 2 b^{1/2}} K_b(x,u) \rd u
\geq c \, b^{-1/2}\cdot \frac{b^{1/2}}{2}
\geq c.
\end{equation}
Combining~\eqref{eq:Ak-lower} and~\eqref{eq:intK-lower} yields
\[
U_k(x) \geq c \, b^{\beta/2}.
\]

Next, set $H_b \leqdef L_{\beta}(3 b^{1/2})^{\beta}$. Since $0\leq \psi \leq 1$ and the translated bumps have disjoint supports, one has
\[
\sup_{u\in [0,1]} f_{\beta,b}(u) = 1 + H_b.
\]
Therefore, whenever $f_{\beta,b}(u) \geq f_{\beta,b}(x)$,
\[
0 \leq f_{\beta,b}(u) - f_{\beta,b}(x) \leq (1 + H_b) - f_{\beta,b}(x)
= H_b \, \left\{1 - \psi\left(\frac{x-t^{(k)}}{3 b^{1/2}}\right)\right\}.
\]
Since $\psi(v)=(1-v^2)^3$ for $|v|\leq 1$, one has $1-\psi(v)=3v^2-3v^4+v^6\leq 3v^2$, and thus, for $x\in T_k(\e,b)$,
\[
1 - \psi\left(\frac{x-t^{(k)}}{3 b^{1/2}}\right)
\leq 3\left(\frac{\e}{3}\right)^2 = \frac{\e^2}{3}.
\]
Hence,
\[
0 \leq f_{\beta,b}(u) - f_{\beta,b}(x) \leq \frac{H_b \, \e^2}{3},
\quad \text{whenever } f_{\beta,b}(u) \geq f_{\beta,b}(x),
\]
and therefore
\[
V(x) \leq \frac{H_b \, \e^2}{3} \int_0^1 K_b(x,u) \rd u \leq \frac{H_b \, \e^2}{3} \leq C \e^2 b^{\beta/2}.
\]
Choosing $\e > 0$ small enough so that $C \e^2 \leq c/2$, we obtain
\[
\Big|\EE_{f_{\beta,b}}[\hat{f}_{n,b}(x)] - f_{\beta,b}(x)\Big| \geq \frac{c}{2} b^{\beta/2}, \quad x\in T_k(\e,b), ~k\in \Delta_N.
\]
Integrating over $x$ and summing over $k\in \Delta_N$ gives
\[
\|\EE_{f_{\beta,b}}[\hat{f}_{n,b}] - f_{\beta,b}\|_{L^1([0,1])}
\geq \sum_{k\in \Delta_N} \int_{T_k(\e,b)} \Big|\EE_{f_{\beta,b}}[\hat{f}_{n,b}(x)] - f_{\beta,b}(x)\Big| \rd x
\geq c \, b^{\beta/2}\sum_{k\in \Delta_N} |T_k(\e,b)|.
\]
Since $|T_k(\e,b)| = 2 \e b^{1/2}$ and $|\Delta_N| = N \asymp b^{-1/2}$, the sum is bounded below by a positive constant independent of $b$. This proves $\|\EE_{f_{\beta,b}}[\hat{f}_{n,b}] - f_{\beta,b}\|_{L^1([0,1])} \geq c \, b^{\beta/2}$ for $b$ small enough.

Finally, since $\widetilde{f}_{\beta,b}$ differs from $f_{\beta,b}$ only on $[7/8,1]$ and, for $x\in[1/4,5/6]$, the kernel mass on $[7/8,1]$ is $\OO(\exp(-c/b))$, the same lower bound holds for $\widetilde{f}_{\beta,b}$ after possibly shrinking $b_1$ and adjusting constants. In particular,
\[
\|\EE_{\widetilde{f}_{\beta,b}}[\hat{f}_{n,b}] - \widetilde{f}_{\beta,b}\|_{L^1([0,1])} \geq c \, b^{\beta/2},
\]
and on a set of Lebesgue measure $1$, $\|g\|_p \geq \|g\|_{L^1([0,1])}$ for every $p\in[1,\infty)$ (see \eqref{eq:Lp.L1}), so
\[
R_n(\hat{f}_{n,b},\widetilde{f}_{\beta,b}) \geq \|\EE_{\widetilde{f}_{\beta,b}}[\hat{f}_{n,b}] - \widetilde{f}_{\beta,b}\|_p \geq \|\EE_{\widetilde{f}_{\beta,b}}[\hat{f}_{n,b}] - \widetilde{f}_{\beta,b}\|_{L^1([0,1])} \geq c \, b^{\beta/2}.
\]
This concludes the proof.
\end{proof}

\end{appendices}

\section*{Funding}
\addcontentsline{toc}{section}{Funding}

Fr\'ed\'eric Ouimet is supported by the Natural Sciences and Engineering Research Council of Canada (NSERC) through Discovery Grant RGPIN-2026-04471 and Discovery Launch Supplement DGECR-2026-00449.

\section*{References}
\addcontentsline{toc}{chapter}{References}



\end{document}